\documentclass [a4paper,12pt]{amsart}
\usepackage{amsthm}
\usepackage{amsmath}
\usepackage{amssymb}
\usepackage[ansinew]{inputenc}
\usepackage{amscd}
\usepackage[ansinew]{inputenc}
\usepackage[mathscr]{euscript}
\usepackage{color}
\usepackage{hyperref}

\newtheorem{thm}{Theorem}[section]
\newtheorem{cor}[thm]{Corollary}
\newtheorem{lem}[thm]{Lemma}
\newtheorem{prop}[thm]{Proposition}

\theoremstyle{definition}
\newtheorem{ex}[thm]{Example}

\theoremstyle{definition}
\newtheorem{defn}[thm]{Definition}

\theoremstyle{definition}
\newtheorem{rem}[thm]{Remark}

\theoremstyle{definition}

\def\Q{\mathbb Q}
\def\C{\mathbb C}
\def\R{\mathbb R}

\def\Z{\mathbb Z}

\def\A{\mathscr A}
\def\B{\mathscr B}
\def\dim{\operatorname{dim}}

\def\supp {\mathrm{supp}}

\def\O{\mathcal O}

\def\m{\mathit m}

\def\V{\mathrm V}

\def\RR{\mathscr R}
\def\LL{\mathscr L}
\def\k{\mathit k}

\def\geq{\geqslant}
\def\leq{\leqslant}

\def\*{\blacksquare}

\subjclass[$2010$ Mathematics Subject Classification]{Primary
32S05; Secondary 13H15}

\begin{document}

\title[{\L}ojasiewicz exponent of families of ideals and Newton filtrations]
{{\L}ojasiewicz exponent of families of ideals, Rees mixed\\ \vskip4pt multiplicities and Newton filtrations}

%----------Author 1
\author{Carles Bivi\`a-Ausina}
\address{
Institut Universitari de Matem\`atica Pura i Aplicada,
Universitat Polit\`ecnica de Val\`encia,
Cam\'i de Vera, s/n,
46022 Val\`encia,
Spain}
\email{carbivia@mat.upv.es}

%----------Author 2
\author{Santiago Encinas}
\address{Departamento de Matem\'atica Aplicada,
Universidad de Valladolid,
Avda. Salamanca s/n,
47014 Valladolid,
Spain}
\email{sencinas@maf.uva.es}


\thanks{The first author was partially supported by DGICYT
Grant MTM2009--08933. The second author was partially supported by
DGICYT Grant MTM2009--07291 and CCG08-UAM/ESP-3928.}

%%%%%%%%%%%%%%%%%%%%%%%%%%%%%%%%%%%%%%%%%%%%%%%%%%%%%%%%%%%%%%%%%%%%Resum
\begin{abstract}
We give an expression for the {\L}ojasiewicz exponent of a wide class of
n-tuples of ideals $(I_1,\dots, I_n)$ in $\O_n$ using the information given by a fixed Newton filtration.
In order to obtain this expression we consider a reformulation of \L ojasiewicz exponents in terms of Rees mixed multiplicities.
As a consequence, we obtain a wide class of
semi-weighted homogeneous functions $(\C^n,0)\to (\C,0)$ for which the \L ojasiewicz of its gradient map
$\nabla f$ attains the maximum possible value.
\end{abstract}

\maketitle
%%%%%%%%%%%%%%%%%%%%%%%%%%%%%%%%%%%%%%%%%%%%%%%%%%%%%%%%%%%%%%%%%%%%%%%%%%%

\section{Introduction}

%%%%%%%%%

Let $\O_n$ be the ring of complex analytic function germs $f:(\C^n,0)\to \C$. S. \L ojasiewicz proved in \cite{Lojasiewicz1959} (as a consequence of a more general result of functional analysis) that if $I$ is an ideal of $\O_n$ of finite colength and $g_1,\dots, g_s$ is a generating system of $I$, then there exists a real number $\alpha>0$ for which there exist a constant $C>0$ and an open neighbourhood $U$ of $0$ in $\C^n$ such that
\begin{equation}\label{Lojineq}
\Vert x\Vert^\alpha\leq C\sup_i\vert g_i(x)\vert,
\end{equation}
for all $x\in U$. The infimum of such $\alpha$ is called the {\it \L ojasiewicz exponent of $I$} and is denoted by $\LL_0(I)$.
If $g:(\C^n,0)\to (\C^m,0)$ denotes a complex analytic map germ such that $g^{-1}(0)=\{0\}$, then the {\it \L ojasiewicz exponent of
$g$} is defined as $\LL_0(g)=\LL_0(I)$, where $I$ denotes the ideal of $\O_n$ generated by the component functions of $g$.
If $f\in\O_n$ has an isolated singularity at the origin, then the \L ojasiewicz exponent of the gradient map $\nabla f:(\C^n,0)\to (\C^n,0)$
is particularly known in singularity theory, by virtue of the result of Teissier \cite[p.\ 280]{Teissier1977} stating that the degree of $C^0$-determinacy of $f$ is equal to $[\LL_0(\nabla f)]+1$, where $[a]$ stands for the integer part of a given $a\in \R$.
It is known that $\LL_0(\nabla f)$ is an analytical invariant of $f$ but it is still unknown if $\LL_0(\nabla f)$ is a topological invariant of $f$ (see for instance \cite{KOP2009}).

Let $\overline\nu_I:\O_n\to \R_{\geq 0}\cup\{+\infty\}$ be the asymptotic Samuel function of $I$ (see \cite{LT2008} or \cite[p.\ 139]{HunekeSwanson2006}). By a result of Nagata \cite{Nagata1957} the range of $\overline \nu_I$ is a subset of $\Q_{\geq 0}\cup\{+\infty\}$.
If $J$ is any ideal of $\O_n$, let us define $\overline\nu_I(J)=\min\{\overline\nu_I(h_1),\dots,\overline\nu_I(h_r)\}$, where
$h_1,\dots, h_r$ denotes any generating system of $J$. We will denote by $m_n$, or simply by $m$, the maximal ideal of $\O_n$.
Lejeune and Teissier proved in \cite{LT2008} the following fundamental facts: $\LL_0(I)=\frac{1}{\overline \nu_I(m)}$ (therefore $\LL_0(I)$ is a rational number), relation (\ref{Lojineq}) holds for $\alpha=\LL_0(I)$, for some constant $C>0$ and some open neighbourhood $U$ of $0\in\C^n$,
and $\LL_0(I)$ is expressed as
\begin{equation}\label{LI}
\LL_0(I)=\min\bigg\{\frac{p}{q}: p,q\in\Z_{\geq 1},\,\m^p\subseteq\overline{I^q}\bigg\},
\end{equation}
where $\overline J$ denotes the integral closure of a given ideal $J$ of $\O_n$. The above expression was one of the motivations that lead the first author
to introduce in \cite{Bivia2009} the notion of \L ojasiewicz exponent of a set of ideals (see Definition \ref{Lojsetofideals}).
By substituting $m$ by a proper ideal $J$ of $\O_n$ in (\ref{LI}) we obtain what is known as the
{\it \L ojasiewicz exponent of $I$ with respect to $J$} (see (\ref{LojaExpGen}), (\ref{LIGeneral}) and \cite{LT2008}).

%%%%%%%%%%%%%%

The effective computation of the \L ojasiewicz exponent $\LL_0(I)$ of a given ideal $I$ of $\O_n$ is a non-trivial problem, since it is intimately related with the determination of the integral closure of $I$. The authors applied in \cite{BiviaEncinas2009} the explicit
construction of a log-resolution of $I$ to show an effective method to compute $\LL_0(I)$.
Newton polyhedra have proven to be a powerful tool in the estimation, and determination in some cases, of \L ojasiewicz exponents,
as can be seen in \cite{Bivia2003}, \cite{Fukui1991}, \cite{Lenarcik} and \cite{Oleksik2009}.

%%%%%%%%%%%%%%

Let $w=(w_1,\dots, w_n)\in\Z^n_{\geq 1}$, we say that a monomial $x_1^{k_1}\cdots x_n^{k_n}$
has {\it $w$-degree $d$} when $w_1k_1+\cdots +w_nk_n=d$. A polynomial function $f:\C^n\to \C$
is said to be {\it weighted homogeneous of degree $d$ with respect to $w$} when $f$ is written as a sum of monomials of $w$-degree $d$.
A function $h\in\O_n$ is termed {\it semi-weighted} homogeneous of degree $d$ with respect to $w$ when $h$ is expressed as a sum $h=h_1+h_2$, where
$h_1$ is weighted homogeneous of degree $d$ with respect to $w$, $h_1$ has an isolated singularity at the origin and $h_2$ is a sum of monomials
 of $w$-degree greater that $d$.

The motivation of our work is the article \cite{KOP2009} of Krasi\'nski-Oleksik-P\l oski, whose main result is a formula for the \L ojasiewicz exponent $\LL_0(\nabla f)$ of any weighted homogeneous function $f:\C^3\to \C$ in terms of the weights and the degree of $f$. More precisely, if $f\in\O_3$ is weighted homogeneous with respect to $(w_1,w_2,w_3)$ of degree $d$ and $w_0=\min\{w_1,w_2,w_3\}$ then it is proven in \cite{KOP2009} that
\begin{equation}\label{formulaKOP}
\LL_0(\nabla f)=\min\left\{\frac{d-w_0}{w_0},\,\,\, \prod_{i=1}^3\left(\frac{d}{w_i}-1\right) \right\}.
\end{equation}
We remark that when $d\geq 2w_i$, for all $i=1,2,3$, then $\LL_0(\nabla f)=\frac{d-w_0}{w_0}$.
As a consequence of (\ref{formulaKOP}) we have that if $f:\C^3\to\C$ is a weighted homogeneous
function with respect to $(w_1,w_2,w_3)$, then
$\LL_0(\nabla f)$ is a topological invariant of $f$,
by the results of Saeki \cite{Saeki1988} and Yau \cite{Yau1988}.

Let us fix a vector of weights $w=(w_1,\dots, w_n)\in \Z^n_{\geq 1}$ and let $w_0=\min_iw_i$.
Let $f\in\O_n$ be a semi-weighted homogeneous function of degree $d$ with respect to $w$. It is well-known that
\begin{equation}\label{Lojmaxim}
\LL_0(\nabla f)\leq \frac{d-w_0}{w_0}.
\end{equation}
If $d<2w_i$, for some $i\in\{1,\dots, n\}$, then it is easy to find examples where inequality (\ref{Lojmaxim}) is strict.
Assuming $d\geq 2w_i$, for all $i=1,\dots, n$, then it is reasonable to conjecture that equality holds in (\ref{Lojmaxim}).

In \cite{BiviaEncinas2011} we considered the problem of finding a sufficient condition on $f$
for equality in (\ref{Lojmaxim}).
We addressed this problem in the framework of \L ojasiewicz exponents
of sets of $n$ ideals in $\O_n$ (in the sense of \cite{Bivia2009}) and weighted homogeneous filtrations. Thus, we introduced in \cite{BiviaEncinas2011} the concept of sets of ideals admitting a {\it $w$-matching} (see Definition \ref{defwmatching}).
The application of this notion to gradient maps lead to determine a wide class of functions for which equality holds in (\ref{Lojmaxim}). In particular, we found that this equality is true for every semi-weighted homogeneous function $f\in\O_n$ of degree $d$ with respect to $w$ such that $w_i$ divides $d$, for all $i=1,\dots, n$ (see \cite[Corollary 4.16]{BiviaEncinas2011}).

In this article we show an extension of the main result of \cite{BiviaEncinas2011} to Newton filtrations in general (see Theorem \ref{main}).
This extension projects to new results about the \L ojasiewicz exponent of the gradient of semi-weighted homogeneous functions.
In this direction, we emphasize that Corollary \ref{sufficient} shows a quite wide class of functions $f\in\O_n$ for which
$\LL_0(\nabla f)$ attains the maximum possible value, that is, such that equality holds in (\ref{Lojmaxim}).
The techniques that we will apply in this article come from multiplicity theory in local rings. More precisely, we use
the notion of mixed multiplicities of a family of ideals of finite colength and its generalization to suitable families of ideals
called Rees mixed multiplicities (see \cite{Bivia2008}).

Let us consider a Newton polyhedron $\Gamma_+$ in $\R^n_+$.
The key ingredient in our approach to \L ojasiewicz exponents in this article is the notion of $\Gamma_+$-linked pairs $(I; J_1,\dots, J_n)$,
where $I, J_1,\dots, J_n$ are ideals of $\O_n$ (see Definition \ref{linkage}). This notion is expressed via the non-degeneracy condition
explored in \cite{BFS}.

%%%%%%%%%%%%%%

The article is organized as follows. In Section \ref{secciobasica} we recall the basic definitions and previous results
(mainly from \cite{Bivia2008} and \cite{Bivia2009}) that lead to the definition of \L ojasiewicz exponent of a set of ideals. For the sake of completeness we also introduce in Section \ref{secciobasica} some auxiliary results needed in the proof of the main result.
In Section \ref{mainsection} we show the main result of the article (see Theorem \ref{main}) and discuss some examples.
In Section \ref{whfiltrations} we particularize the techniques developed in Section \ref{mainsection} to weighted homogeneous filtrations and, as said before, we derive new results about the \L ojasiewicz exponent of gradient maps.

%%%%%%%%%%%%%%%%%%%%%%%%%%%%%%%%%%%%%%%%%%%%%%%%%%%%%%%%%%%%%%%%%%%%%%%%%%%%%%%%%%%%%%%%%%%%%%%%%%%%%%%%%%%%%%%%%%
\section{The {\L}ojasiewicz exponent of a set of ideals}\label{secciobasica}

Let  $(R,m)$ be a Noetherian local ring of dimension $n$ and let $I$ be an ideal of $R$
of finite colength (also called an $m$-primary ideal). Then we denote by $e(I)$ the Samuel
multiplicity of $I$ (see \cite{HIO}, \cite[\S 11]{HunekeSwanson2006} or \cite{Teissier1973} for the definition and basic properties of this notion).
We recall that $e(I)=\ell(R/I)$ if $I$ admits a generating system formed by $n$ elements.

If $I_1,\dots, I_n$ are ideals of
$R$ of finite colength, then we denote by $e(I_1,\dots,I_n)$
the mixed multiplicity of $I_1,\dots, I_n$ in the sense of Risler and Teissier \cite{Teissier1973}
(see also \cite[\S 17]{HunekeSwanson2006} or \cite{TeissierMSRI}).

\begin{defn}\cite{Bivia2008}\label{lasigma}
Let $(R,m)$ be a Noetherian local ring of dimension $n$. Let
$I_1,\dots,
I_n$ be ideals of $R$. Then we define the {\it Rees mixed
multiplicity of $I_1,\dots, I_n$ as}
\begin{equation}\label{sigma}
\sigma(I_1,\dots, I_n)=\max_{r\in\Z_+}\,e(I_1+\m^r,\dots,
I_n+\m^r),
\end{equation}
when the number on the right hand side is finite. If the set of
integers
$\{e(I_1+\m^r,\dots, I_n+\m^r): r\in\Z_+\}$ is non-bounded then we
set $\sigma(I_1,\dots, I_n)=\infty$.
\end{defn}

We remark that if $I_i$ is an ideal of $R$ of finite colength, for all
$i=1,\dots, n$, then $\sigma(I_1,\dots, I_n)=e(I_1,\dots, I_n)$.
Moreover, if $I_1=\cdots=I_n=I$, for some ideal $I$ of $R$ of finite colength, then
$e(I_1,\dots, I_n)=e(I)$.

Let us suppose that the residue field $k=R/m$ is infinite. Let
$I_1,\dots, I_n$ be ideals of $R$ and let $\{a_{i1},\dots, a_{is_i}\}$ be
a generating system of $I_i$, where $s_i\geqslant 1$, for
$i=1,\dots, n$. Let $s=s_1+\cdots +s_n$. We say that a property
holds for {\it sufficiently general} elements of $I_1\oplus \cdots
\oplus I_n$ if there exists a non-empty Zariski-open set $U$ in
$k^s$ such that the said property holds for all elements
$(g_1,\dots, g_n)\in I_1\oplus \cdots \oplus I_n$ for which
$g_i=\sum_{j}u_{ij}a_{ij}$, $i=1,\dots, n$, where $(u_{11}, \dots,
u_{1s_1},\dots, u_{n1}, \dots, u_{ns_n})\in U$.

The next proposition characterizes the finiteness of $\sigma(I_1,\dots, I_n)$.

\begin{prop}\cite[p.\,393]{Bivia2008}\label{sigmaexists}
Let $I_1,\dots, I_n$ be ideals of a Noetherian
local ring $(R,m)$ such that the residue field $\k=R/m$ is
infinite. Then $\sigma(I_1,\dots, I_n)<\infty$ if and only if
there exist elements $g_i\in I_i$, for $i=1,\dots, n$, such that
$\langle g_1,\dots, g_n\rangle$ has finite colength. In this case,
we have that $\sigma(I_1,\dots, I_n)=e(g_1,\dots, g_n)$ for
sufficiently general elements $(g_1,\dots, g_n)\in I_1\oplus
\cdots \oplus I_n$.
\end{prop}

If $I$ and $J$ are ideals of finite colength of $R$ such that $J\subseteq I$
then it is well-known that $e(J)\geq e(I)$ (see for instance \cite[p.\ 292]{Teissier1973}). The following result extends this inequality to Rees mixed multiplicities.

\begin{lem}\label{reverseincl}\cite[p. 392]{Bivia2009}
Let $(R,m)$ be a Noetherian local ring of dimension $n\geq 1$. Let
$J_1,\dots, J_n$ be ideals of $R$ such that
$\sigma(J_1,\dots, J_n)<\infty$. Let $I_1,\dots, I_n$ be ideals of
$R$ for which $J_i\subseteq I_i$, for all $i=1,\dots, n$. Then
$\sigma(I_1,\dots, I_n)<\infty$ and
$$
\sigma(J_1,\dots, J_n)\geqslant \sigma(I_1,\dots, I_n).
$$
\end{lem}

Let us recall some basic definitions. We will denote by $\R_+$ the set of non-negative
real numbers. We also set $\Z_+=\Z\cap\R_+$. Let us fix a coordinate system
$x_1,\dots,x_n$ in $\C^n$.  If $k=(k_1,\dots, k_n)\in\Z^n_+$, we
will denote the monomial $x_1^{k_1}\cdots x_n^{k_n}$ by $x^k$.
Let $A\subseteq \Z^n_+$, then the {\it Newton polyhedron determined by $A$}, denoted by $\Gamma_+(A)$, is
the convex hull of the set $\{k+v: k\in A, v\in\R^n_+\}$. A subset $\Gamma_+\subseteq \R^n_+$
is called a {\it Newton polyhedron} when $\Gamma_+=\Gamma_+(A)$, for some $A\subseteq \Z^n_+$. A Newton polyhedron
$\Gamma_+\subseteq \R^n_+$ is termed {\it convenient} when $\Gamma_+$ meets the $x_i$-axis in a point
different from the origin, for all $i=1,\dots, n$.

If $h\in\O_n$ and $h=\sum_ka_kx^k$ denotes the Taylor expansion of $h$
around the origin, then the {\it support of $h$} is the set
$\supp(h)=\{k\in\Z_+^n:a_k\neq 0\}$.  If $h\neq 0$, then the {\it Newton
polyhedron of $h$} is defined as $\Gamma_+(h)=\Gamma_+(\supp(h))$. If $h=0$ then we set
$\Gamma_+(h)=\emptyset$.  If $I$ is an ideal of $\O_n$ and
$g_1,\dots,g_s$ is a generating system of $I$, then we define the {\it
Newton polyhedron of $I$} as the convex hull of $\Gamma_+(g_1)\cup
\cdots \cup\Gamma_+(g_s)$.  It is easy to check that the definition of
$\Gamma_+(I)$ does not depend on the chosen generating system of $I$. We denote
the Newton boundary of $\Gamma_+(I)$ by $\Gamma(I)$.

We say that a proper ideal $I$ of $\O_n$ is {\it monomial} when $I$ admits a
generating system formed by monomials. We recall that if $I$ is a monomial ideal of $\O_n$ of finite
colength, then $e(I)=n!\V_n(\R^n_+\smallsetminus\Gamma_+(I))$, where $\V_n$ denotes $n$-dimensional volume
(see for instance \cite[p. 239]{TeissierMSRI}).

\begin{defn}\label{S(I)}
Let $I_1,\dots, I_n$ be monomial ideals of $\O_n$ with
$\sigma(I_1,\dots, I_n)<\infty$.  Then we denote by $\mathcal
S(I_1,\dots, I_n)$ the family of maps $g=(g_1,\dots,
g_n):(\C^n,0)\to (\C^n,0)$ such that $g^{-1}(0)=\{0\}$, $g_i\in I_i$,
for all $i=1,\dots, n$, and $\sigma(I_1,\dots, I_n)=e(g_1,\dots,
g_n)$, where $e(g_1,\dots, g_n)$ stands for the multiplicity of the
ideal of $\O_n$ generated by $g_1,\dots, g_n$.  The elements of
$\mathcal S(I_1,\dots, I_n)$ are characterized in
\cite[Theorem 3.10]{Bivia2008}.

We denote by $\mathcal S_0(I_1,\dots, I_n)$ the set formed by the maps
$(g_1,\dots, g_n)\in \mathcal S(I_1,\dots, I_n)$ such that
$\Gamma_+(g_i)=\Gamma_+(I_i)$, for all $i=1,\dots, n$.
\end{defn}

Let $I_1,\dots, I_n$ be ideals of a local ring $(R,m)$ for which
$\sigma(I_1,\dots, I_n)<\infty$.  Then we define
\begin{equation}\label{laerre}
r(I_1,\dots, I_n)=\min\big\{r\in\Z_+: \sigma(I_1,\dots,
I_n)=e(I_1+m^r,\dots, I_n+m^r)\big\}.
\end{equation}

We recall that if $g:(\C^n,0)\to (\C^m,0)$ is an analytic map germ such that $g^{-1}(0)=\{0\}$, then
$\LL_0(g)$ denotes the \L ojasiewicz exponent of the ideal generated by the components of $g$.

\begin{thm}\textnormal{\cite[p.\ 398]{Bivia2009}}\label{base}
Let $I_1,\dots, I_n$ be monomial ideals of $\O_n$ such that
$\sigma(I_1,\dots,
I_n)$ is finite. If $g\in\mathcal S_0(I_1,\dots, I_n)$, then $\mathscr
L_0(g)$ depends only on $I_1,\dots, I_n$ and it is given by
\begin{equation}\label{intrinsec}
\mathscr L_0(g)=\min_{s\geqslant 1}\frac{r(I_1^s,\dots,
I_n^s)}{s}.
\end{equation}
\end{thm}

The previous result motivated the following definition.

\begin{defn}\textnormal{\cite[p.\ 399]{Bivia2009}}\label{Lojsetofideals}
Let $(R,m)$ be a Noetherian local ring of dimension $n$. Let
$I_1,\dots,
I_n$ be ideals of $R$ for which $\sigma(I_1,\dots,
I_n)<\infty$. We define the {\it {\L}ojasiewicz exponent of
$I_1,\dots, I_n$} as
$$
\mathscr L_0(I_1,\dots, I_n)=\inf_{s\geq 1}\frac{r(I_1^s,\dots,
I_n^s)}{s}.
$$
\end{defn}

As a consequence of Lemma \ref{rpowers}, we have that
$r(I_1^s,\dots, I_n^s)\leq s r(I_1,\dots, I_n)$, for all $s\in \Z_{\geq 1}$.
Hence $\LL_0(I_1,\dots, I_n)\leq r(I_1,\dots, I_n)$.

The {\L}ojasiewicz exponent given in Definition \ref{Lojsetofideals} is
coherent with the original definition of {\L}ojasiewicz exponent for an
analytic map (see (\ref{Lojineq})), as is shown in the following result.

\begin{lem}\label{soniguals} Let $g=(g_1,\dots, g_n):(\C^n,0)\to (\C^n,0)$ be an analytic map germ such that $g^{-1}(0)=\{0\}$. Then
\begin{equation}\label{coincideixen}
\LL_0(g)=\LL_0(\langle g_1\rangle,\dots, \langle g_n\rangle).
\end{equation}
\end{lem}

\begin{proof} Let us fix integers $r,s\geq 1$. Let $I$ denote the ideal generated by the components of $g$.
It can be proved using Proposition \ref{sigmaexists} and Lemma \ref{reverseincl} that
$$
e(\langle g_1^s\rangle+m^r,\dots, \langle g_n^s\rangle+m^r)=e(\langle g_1^s,\dots, g_n^s\rangle+m^r).
$$
Moreover $\overline{I^s}=\overline{\langle g_1^s, \dots, g_n^s\rangle}$ (see for instance \cite[p.\ 344]{HunekeSwanson2006}). Then,
by the Rees' multiplicity theorem (see \cite[p.\  222]{HunekeSwanson2006}) we obtain that
$e(g_1^s,\dots, g_n^s)=e(\langle g_1^s\rangle+m^r,\dots, \langle g_n^s\rangle+m^r)$ if and only if $m^r\subseteq \overline{I^s}$.
Hence we deduce that
$$
r(\langle g_1^s\rangle, \dots, \langle g_n^s\rangle)=\min\left\{r: m^r\subseteq \overline{I^s}\right\}.
$$
Therefore
\begin{align*}
\LL_0(\langle g_1\rangle,\dots, \langle g_n\rangle)&=\min_{s\geq 1}\frac{\min\left\{r: m^r\subseteq \overline{I^s}\right\}}{s}\\
&=\min\bigg\{\frac pq: p,q\in\Z_{\geq 1},\,\, m^p\subseteq \overline{I^q}\bigg\}=\LL_0(I),
\end{align*}
where the last equality follows from (\ref{LI}).
\end{proof}

Under the hypothesis of Definition \ref{Lojsetofideals}, let us denote
by $J$ a proper ideal of $R$. An easy application of Lemma \ref{reverseincl} shows that
$$
\sigma(I_1,\dots, I_n)=\max_{r\in\Z_+}\,\sigma(I_1+J^r,\dots, I_n+J^r).
$$
Hence, let us define
\begin{equation}\label{laerresubJ}
r_J(I_1,\dots, I_n)=\min\big\{r\in\Z_+: \sigma(I_1,\dots,
I_n)=\sigma(I_1+J^r,\dots, I_n+J^r)\big\}.
\end{equation}

Let $I$ be an ideal of $R$ of finite colength. Then we denote by
$r_J(I)$ the number $r_J(I,\dots, I)$, where $I$ is repeated $n$ times.
We deduce from the Rees' multiplicity theorem (see \cite[p.\ 222]{HunekeSwanson2006}) that if $R$ is
quasi-unmixed then $r_J(I)=\min\{r\geq 1: J^r\subseteq\overline I\}$.

\begin{lem}\label{rpowers}\cite[p.\ 581]{BiviaEncinas2011}
Let $(R,m)$ be a Noetherian local ring of dimension $n$. Let
$I_1,\dots,
I_n$ be ideals of $R$ such that $\sigma(I_1,\dots, I_n)<\infty$ and
let $J$ be a proper ideal of $R$. Then
\begin{align*}
r_J(I_1^s,\dots, I_n^s)&\leq sr_J(I_1,\dots, I_n)\\
r_{J^s}(I_1,\dots, I_n)&\geq \frac{1}{s}r_J(I_1,\dots, I_n)
\end{align*}
for any integer $s\geq 1$.
\end{lem}

We remark that the previous lemma was proven in \cite{BiviaEncinas2011} under the assumption that
the ideal $J$ has finite colength, but the same proof works equally for any proper ideal $J$ of $\O_n$.

If $I$ is an ideal of $\O_n$ of finite colength and $J$ is a proper ideal of $\O_n$, then
the {\it \L ojasiewicz exponent of $I$ with respect to $J$}, denoted by $\LL_J(I)$, is defined
as the infimum of those $\alpha>0$ such that there exist a
constant $C>0$ and an open neighbourhood $U$ of $0\in\C^n$ for which
\begin{equation} \label{LojaExpGen}
\sup_j\vert h_j(x)\vert^\alpha\leq C\sup_i\vert g_i(x)\vert,
\end{equation}
for all $x\in U$, where $\{h_j: j=1,\dots, r\}$ and $\{g_i: i=1,\dots, s\}$ are
generating systems of $J$ and $I$, respectively. As a consequence of \cite[\S 7]{LT2008} we have that
$\LL_J(I)$ is a rational number and
\begin{equation}\label{LIGeneral}
\LL_{J}(I)=\min\bigg\{\frac{p}{q}: p,q\in\Z_{\geq 1},\,J^p\subseteq\overline{I^q}\bigg\}.
\end{equation}

If $g:(\C^n,0)\to (\C^m,0)$ is an analytic map germ such that $g^{-1}(0)=\{0\}$ and $J$ is a proper ideal of $\O_n$
then we denote by $\LL_J(g)$ the {\L}ojasiewicz exponent $\LL_J(I)$, where $I$ is the ideal
generated by the component functions of $g$.

Now we extend Definition \ref{Lojsetofideals} by considering $r_J(I_1,\dots, I_n)$ instead of $r(I_1,\dots, I_n)$.

\begin{defn}\label{DefLJI}\cite[p.\ 581]{BiviaEncinas2011}
Let $(R,m)$ be a Noetherian local ring of dimension $n$. Let $I_1,\dots,
I_n$ be ideals of $R$ such that $\sigma(I_1,\dots, I_n)<\infty$. Let
$J$ be a proper ideal of $R$.
We define the {\it {\L}ojasiewicz exponent of $I_1,\dots, I_n$ with
respect to $J$}, denoted by $\LL_J(I_1,\dots, I_n)$, as
\begin{equation}\label{LJI}
\LL_J(I_1,\dots, I_n)=\inf_{s\geq 1}\frac{r_J(I_1^s,\dots, I_n^s)}{s}.
\end{equation}
\end{defn}

Under the conditions of the previous definition, we observe that
$\LL_{J}(I_{1},\ldots,I_{n})$ is expressed as a limit inferior (see \cite[p.\ 581]{BiviaEncinas2011} for details), that is:
\begin{equation*}
    \LL_{J}(I_{1},\ldots,I_{n})=
    \liminf_{s\to\infty}\frac{r_J(I_1^s,\dots, I_n^s)}{s}.
\end{equation*}

If $I$ is an $m$-primary ideal of $R$, then we denote by $\LL_J(I)$
the number $\LL_J(I,\dots, I)$, where $I$ is repeated $n$ times. We remark that
when $R$ is quasi-unmixed then $\LL_J(I)$ is given by
\begin{equation}\label{LIGeneral2}
\LL_{J}(I)=\min\bigg\{\frac{p}{q}: p,q\in\Z_{\geq 1},\,J^p\subseteq\overline{I^q}\bigg\},
\end{equation}
by a direct application of the Rees' multiplicity theorem.

We point out that we are denoting $\LL_m(I_1,\dots, I_n)$ by $\LL_0(I_1,\dots, I_n)$ and that
$\LL_J(I_1,\dots, I_n)$ is not defined when $J$ is the zero ideal. If $I_i$ is an ideal of $\O_n$,
for all $i=1,\dots, n$, then the subscript in $\LL_0(I_1,\dots, I_n)$
corresponds to the commonly used notation to refer to \L ojasiewicz exponents in a neighbourhood of $0\in\C^n$, as defined
in (\ref{Lojineq}). If $J$ is a proper ideal of $\O_n$, we also remark that the result analogous to Lemma \ref{soniguals} obtained by writing
$\LL_0$ instead of $\LL_J$ in equality (\ref{coincideixen}) also holds; it follows by a straightforward reproduction of
the proof of Lemma \ref{soniguals} consisting of replacing $m$ by $J$.

For the sake of completeness, we recall the following two results, which will be applied in the
next section.

\begin{lem}\label{transit}\cite[p.\ 582]{BiviaEncinas2011}
Let $(R,m)$ be a quasi-unmixed Noetherian local ring of dimension
$n$. Let $I_1,\dots,
I_n$ be ideals of $R$ for which $\sigma(I_1,\dots, I_n)<\infty$. If
$J_1,J_2$ are proper ideals of $R$ such that $J_2$ has finite colength then
$$
\LL_{J_1}(I_1,\dots, I_n) \leq
\LL_{J_1}(J_2)\LL_{J_2}(I_1,\dots,I_n).
$$
\end{lem}

\begin{prop}\label{uppers} Let $(R,m)$ be a Noetherian
local ring of dimension $n$. Let $J$ be a proper ideal of $R$.
For each $i=1,\dots, n$ let us consider ideals $I_i$ and $J_i$ of $R$
such that $I_i\subseteq J_i$ and $\sigma(I_1,\dots,
I_n)=\sigma(J_1,\dots,J_n)<\infty$. Then
\begin{equation}\label{monot}
\LL_J(I_1,\dots, I_n)\leqslant \LL_J(J_1,\dots, J_n),
\end{equation}
\end{prop}

\begin{proof}
It follows by replacing $m$ by $I$ in the proof of \cite[Proposition 4.7]{Bivia2009}
\end{proof}

We remark that if $(R,m)$ is a Noetherian quasi-unmixed
local ring of dimension $n$ and $I_1,I_2$ are ideals of
$R$ of finite colength such that $I_1\subseteq I_2$ then $\LL_J(I_1)\geq \LL_J(J_2)$,
as a consequence of (\ref{LIGeneral2}). However, the analogous inequality for \L ojasiewicz exponents of sets of ideals
does not hold in general, as the following example shows.

\begin{ex}
Let us consider the ideals $I_1,I_2$ of $\O_2$ defined by $I_1=\langle x^3\rangle$
and $I_2=\langle y^3\rangle$. Let $J_1,J_2$ be the ideals of $\O_2$
defined by $J_1=\langle x^3, xy\rangle$, $J_2=I_2$. Then we have that $I_i\subseteq J_i$, $i=1,2$,
but $\LL_0(I_1,I_2)=3$ and $\LL_0(J_1,J_2)=6$.
\end{ex}

%%%%%%%%%%%%%%%%%%%%%%%%%%%%%%%%%%%%%%%%%%%%%%%%%%%%%%%%%%%%%%%%%%%%%%%%%%%%%%%%%%%%%%%%%%%%%%%%%%%%%%%%%%%%%%%%%%
\section{Newton filtrations}\label{mainsection}

Let us fix a Newton polyhedron $\Gamma_+\subseteq\R^n_+$. If $v\in\R^n_+\smallsetminus\{0\}$ then we define
\begin{align*}
\ell(v,\Gamma_+)&=\min\{\langle v,k\rangle:k\in\Gamma_+\}\\
\Delta(v,\Gamma_+)&=\big\{k\in\Gamma_+: \langle
v,k\rangle=\ell(v,\Gamma_+)\big\},
\end{align*}
where $\langle \, ,\rangle$ stands for the standard scalar product in $\R^n$.
A {\it face} of $\Gamma_+$ is any set of the form $\Delta(v,\Gamma_+)$, for some $v\in\R^n_+\smallsetminus\{0\}$.
Hence we also say that $\Delta(v,\Gamma_+)$ is the {\it face of $\Gamma_+$
supported by $v$}. The {\it dimension} of a face $\Delta$ of $\Gamma_+$ is the minimum of the dimensions
of the affine subspaces of $\R^n$ containing $\Delta$. If $\Delta$ is a face of $\Gamma_+$ of dimension $n-1$ then we say that $\Delta$ is a {\it facet} of $\Gamma_+$.

It is easy to observe that a face $\Delta$ of $\Gamma_+$ is compact if and only if it is supported
by a vector $v\in (\R_+\smallsetminus\{0\})^n$. The union of all compact faces
of $\Gamma_+$ will be denoted by $\Gamma$; this is also known as the {\it Newton boundary of $\Gamma_+$}. We remark
that $\Gamma_+$ determines and is determined by $\Gamma$, since $\Gamma_+=\Gamma+\R^n_+$.

We denote by $\Gamma_-$ the union of all segments joining the origin and some point of $\Gamma$. Therefore $\Gamma_-$ is a compact subset of $\R^n_+$.
If $\Gamma_+$ is convenient, then $\Gamma_-$ is equal to the closure of $\R^n_+\smallsetminus \Gamma_+$.

If $\Delta$ is a face of $\Gamma_+$, then $C(\Delta)$ denotes the cone formed by all half-rays emanating from the
origin and passing through some point of $\Delta$.

We say that a vector $v\in\Z_+^n\smallsetminus\{0\}$ is {\it
primitive} when the non-zero coordinates of $v$ are mutually prime
integer numbers. Then any facet of $\Gamma_+$ is
supported by a unique primitive vector of $\Z^n_+$. Let us denote by $\mathscr
F(\Gamma_+)$ the set of primitive vectors of $\R^n_+$ supporting
some facet of $\Gamma_+$ and by $\mathscr F_c(\Gamma_+)$
the set of vectors $v\in\mathscr F(\Gamma_+)$ such that $\Delta(v,\Gamma_+)$ is compact.
Let us remark that if $\Gamma_+$ is convenient then $\mathscr F(\Gamma_+)=\mathscr F_c(\Gamma_+)\cup\{e_1,\dots, e_n\}$,
where $e_1,\dots, e_n$ is the canonical basis of $\R^n$.

Let us suppose that $\mathscr F_c(\Gamma_+)=\{v^1,\dots, v^r\}$. Therefore $\ell(v^i,\Gamma_+)\neq 0$, for all
$i=1,\dots, r$. Let us denote by $M_\Gamma$ the least common
multiple of the set of integers $\{\ell(v^1,\Gamma_+), \dots,
\ell(v^r,\Gamma_+)\}$. Hence we define the {\it filtrating map} associated to $\Gamma_+$ as the map
$\phi:\R^n_+\to \R_+$ given by
$$
\phi_\Gamma(k)=\min\bigg\{\frac{M_\Gamma}{\ell(v^i,\Gamma_+)}\langle
k, v^i\rangle: i=1,\dots, r\bigg\},\quad \mathrm{ for\,\, all}\,\,
k\in\R^n_+.
$$

We observe that $\phi_\Gamma(\Z^n_+)\subseteq\Z^n_+$, $\phi_\Gamma(k)=M_\Gamma$, for all $k\in\Gamma$, and the map $\phi_\Gamma$ is linear on each
cone $C(\Delta)$, where $\Delta$ is any compact face of $\Gamma_+$.

Let us define the map $\nu_\Gamma:\O_n\to \R_+\cup\{+\infty\}$ by $\nu_\Gamma(h)=\min\{\phi_\Gamma(k): k\in\supp(h)\}$, for all
$h\in\O_n$, $h\neq 0$; we set $\nu_\Gamma(0)=+\infty$. We refer to $\nu_\Gamma$ as the {\it
Newton filtration} induced by $\Gamma_+$ (see also \cite{BFS, K}).

From now on, we will assume that $\Gamma_+$ is a convenient Newton polyhedron in $\R^n_+$.

Let $h\in\O_n$ and let $h=\sum_ka_kx^k$ be the Taylor expansion of $h$ around the origin. If $A$ is a compact subset of $\R^n$ then we denote by
$h_A$ the sum of all terms $a_kx^k$ such that $k\in A$. If $\supp(h)\cap A=\emptyset$, then we set $h_A=0$. Let $J$ be an ideal of $\O_n$ and let $g_1,\dots, g_s$ be a generating system of $J$. We recall that $J$ is said to be {\it Newton non-degenerate} (see \cite{Bivia2005} or \cite{Saia1996}) when
$$
\big\{x\in\C^n: (g_1)_\Delta(x)=\cdots=(g_s)_\Delta(x)=0\big\}\subseteq \big\{x\in\C^n: x_1\cdots x_n=0\big\},
$$
as set germs at $0\in\C^n$, for each compact face $\Delta$ of $\Gamma_+(J)$ (see Theorem \ref{caractdelta}). It is immediate to check that this definition does not depend on the chosen generating system of $J$. In particular, any monomial ideal is Newton non-degenerate.

The next result compares the asymptotic Samuel function and the Newton filtration.

\begin{prop}\label{Impa}\cite[p.\ 26]{Bivia2005} Let $J\subseteq \O_n$ be an ideal of finite colength.
Let $\Gamma$ denote the Newton boundary of $\Gamma_+(J)$ and let $M=M_\Gamma$. Then $M\overline
\nu_J\leq \nu_\Gamma$ and equality holds if and only if $J$ is Newton non-degenerate.
\end{prop}

As a consequence of the previous result, if $J$ is an ideal of finite colength of $\O_n$ and
$r_i=\min\{r:re_i\in\Gamma_+(J)\}$, for all $i=1,\dots, n$,
then $\max\{r_1,\dots, r_n\}\leq \LL_0(J)$ and equality holds if $J$ is a Newton non-degenerate ideal (see \cite[p.\ 27]{Bivia2005} for details).

Given an integer $r\in\Z_{\geq 0}$, we denote by $\mathscr A_r$ the ideal of $\O_n$ generated by the elements
$h\in\O_n$ such that $\nu_\Gamma(h)=r$ (we assume that the ideal generated by the empty set is $0$).
In particular, $\A_{M_\Gamma}=\langle x^k:k\in\Gamma_+\rangle$.

Moreover, we denote by $\B_r$ the ideal of $\O_n$ generated by the elements $h\in\O_n$ for which $\nu_\Gamma(h)\geq r$.
Then
$$
\mathscr B_r=\big\{h\in\O_n: \phi_\Gamma(\supp(h))\subseteq [r,+\infty[\,\big\}\cup\{0\},
$$
for all $r\geq 0$ and $\nu_\Gamma(h)=\max\{r\geq 0: h\in \mathscr B_r\}$, for all $h\in\O_n$, $h\neq 0$.
We will refer indistinctly to the map $\nu_\Gamma$ and to the family of ideals $\{\B_r\}_{r\geq 1}$
as the {\it Newton filtration induced by $\Gamma_+$}.

It is immediate to check that
\begin{enumerate}
\item[(a)] $\B_r$ is an integrally closed monomial ideal of finite colength, for all $r\geq 1$;
\item[(b)] $\B_r\B_s\subseteq \B_{r+s}$, for all $r,s\geq 1$;
\item[(c)] $\B_0=\O_n$.
\end{enumerate}

If $I$ is an ideal of $\O_n$, then we denote by $\nu_\Gamma(I)$ the maximum of those $r$ such that $I\subseteq \B_r$. Then,
if $g_1,\dots, g_s$ denotes any generating system of $I$, we have
$$
\nu_\Gamma(I)=\min\{\nu_\Gamma(g_1),\dots, \nu_\Gamma(g_s)\}.
$$

Given an integer $r\geq 0$, we observe that $\A_r\subseteq \B_r$ and $\overline {\A_r}\neq
\B_r$ in general. Moreover it follows easily that $\overline
{\A_r}=\B_r$ if and only if $\A_r$ is an ideal of finite colength of
$\O_n$.

Let us remark that $\supp(\A_{M_{\Gamma}})=\Gamma_+\cap\Z^n_+$, $\A_{M_{\Gamma}}$ has
finite colength and
$e(\A_{M_{\Gamma}})=n!\V_{n}(\Gamma_{-})$, since $\Gamma$ is convenient and $\A_{M_{\Gamma}}$ is a monomial ideal (see
the paragraph before Definition \ref{S(I)}).

\begin{prop}\label{BezoutlikeNewton}
Let us fix a family of ideals $J_1,\dots, J_n$ of $\O_n$ such that $\sigma(J_1,\dots, J_n)<\infty$. Let
$\nu_\Gamma(J_i)=r_i$, for all $i=1,\dots, n$, and let $M=M_\Gamma$. Then
\begin{equation}\label{nondegongamma}
\sigma(J_1,\dots, J_n)\geq\frac{r_1\cdots r_n}{M^n}n!\V_n(\Gamma_-).
\end{equation}
\end{prop}

\begin{proof} By Proposition \ref{sigmaexists} we have that
$\sigma(J_1,\dots, J_n)=e(g_1,\dots, g_n)$, for a sufficiently general element $(g_1,\dots, g_n)\in J_1\oplus\cdots\oplus J_n$. Then
the result arises as a direct application of \cite[Theorem 3.3]{BFS}.
\end{proof}

As a consequence of \cite[Theorem 3.3]{BFS}, equality in (\ref{nondegongamma}) is characterized by means of a condition imposed to
any element $(g_1,\dots, g_n)\in J_1\oplus \cdots \oplus J_n$ such that $e(g_1,\dots, g_n)=\sigma(J_1,\dots, J_n)$ (we refer the reader to \cite{BFS} for details). By coherence with the nomenclature of \cite[Theorem 3.3]{BFS} we introduce the following definition.

\begin{defn}\label{idealsnodeg}
Let $J_1,\dots, J_n$ be a family of ideals of $\O_n$ such that $\sigma(J_1,\dots, J_n)<\infty$. Let $M=M_\Gamma$. We say that $(J_1,\dots, J_n)$ is {\it non-degenerate on $\Gamma_+$}, or that $(J_1,\dots, J_n)$ is {\it $\Gamma_+$-non-degenerate}, when equality holds in (\ref{nondegongamma}). That is, when $\sigma(J_1,\dots, J_n)=\frac{r_1\cdots r_n}{M^n}e(\A_M)$.
\end{defn}

Under the hypothesis of the previous definition, let us suppose that $J_1$ is principal, that is, $J_1=\langle h\rangle$, for some $h\in\O_n$.
Then, in order to simplify the notation, we will write $(h,J_2,\dots, J_n)$ instead of $(\langle h\rangle, J_2,\dots, J_n)$. We will adopt the same simplification if any other ideal $J_i$ is principal, for some $i\in\{1,\dots, n\}$. Hence, the previous definition applies to
germs of complex analytic maps $(g_1,\dots, g_n):(\C^n,0)\to (\C^n,0)$ such that $g^{-1}(0)=\{0\}$.

If $r_1,\dots, r_n\in\Z_{\geq 1}$, then it is not true in general
that $\sigma(\A_{r_1},\dots, \A_{r_n})<\infty$, even if $\A_{r_i}\neq
0$, for all $i=1,\dots, n$. However $\sigma(\B_{r_1},\dots,
\B_{r_n})<\infty$, since $\B_{r_i}$ has finite colength, for all
$i=1,\dots, n$. If $\sigma(\A_{r_1},\dots, \A_{r_n})<\infty$, then it is also not true
in general that $(\A_{r_1},\dots, \A_{r_n})$ is non-degenerate on $\Gamma_+$,
as the next example shows. If $\Gamma_+$ has only one compact face of dimension $n-1$ (that is,
if the Newton filtration induced by $\Gamma_+$ is a weighted homogeneous filtration), then the condition
$\sigma(\A_{r_1},\dots, \A_{r_n})<\infty$ implies that $(\A_{r_1},\dots, \A_{r_n})$ is non-degenerate on $\Gamma_+$ (see
\cite[Proposition 4.2]{BiviaEncinas2011} for details).

\begin{ex} Let $J=\langle x^4, xy, y^4\rangle$ and let $\Gamma_+=\Gamma_+(J)$. We observe that $M_\Gamma=4$ and the map $\phi_\Gamma:\R^2_+\to \R$ is given by $\phi_\Gamma(k_1,k_2)=\min\{k_1+3k_2, 3k_1+k_2\}$, for all $(k_1,k_2)\in\R^2_+$.
Hence we have that $\A_5=\langle x^5, x^2y, xy^2, y^5\rangle$ and
$$
\sigma(\A_5,\A_5)=e(\A_5)=13\neq \frac{5\cdot 5}{4^2}e(J)=\frac{200}{16}=\frac{25}{2}.
$$
\end{ex}

For the sake of completeness, we show in Proposition \ref{revisio} a reformulation of \cite[Theorem 3.3]{BFS} considering the notion of Rees mixed multiplicity.
If $\Delta$ is a compact face of $\Gamma_+$, then we denote by $\RR_\Delta$ the subring of $\O_n$ formed by all germs $h\in\O_n$ such that
$\supp(h)\subseteq C(\Delta)$. If $\alpha>0$, then $\alpha\Delta$ will denote the set $\{\alpha k:k\in \Delta\}$.
Given a function germ $h\in\O_n$, if $h=\sum_ka_kx^k$ is the Taylor expansion of $h$ around the origin, then we denote by
$p_\Delta(h)$ the sum of all terms $a_kx^k$ for which $\nu_\Gamma(x^k)=\nu_\Gamma(h)$ and $k\in C(\Delta)$. If
no such terms exist, then we set $p_\Delta(h)=0$.
We recall that $h_\Delta$ denotes the sum of all terms $a_kx^k$ such that $k\in \Delta$.
Hence, if $d=\nu_\Gamma(h)$, we observe that
\begin{equation}\label{initial}
p_\Delta(h)=h_{\frac{d}{M}\Delta}.
\end{equation}

\begin{prop}\label{revisio} Let $g=(g_1,\dots, g_n):(\C^n,0)\to (\C^n,0)$ be a complex analytic map. Then the following conditions are equivalent:
\begin{enumerate}
\item $g$ is non-degenerate on $\Gamma_+$;
\item for each compact facet $\Delta$ of $\Gamma_+$, the ideal of $\RR_\Delta$ generated by $p_\Delta(g_1),\dots, p_\Delta(g_n)$ has finite colength in $\RR_\Delta$.
\end{enumerate}
\end{prop}

\begin{proof} Condition (1) means that $e(g_1,\dots, g_n)=\frac{r_1\cdots r_n}{M^n}n!\V_n(\Gamma_-)$, where $r_i=\nu_\Gamma(g_i)$, for all $i=1,\dots, n$. Therefore the result is an immediate consequence of \cite[Theorem 3.3]{BFS}.
\end{proof}

We remark that the equivalence of the previous result is considered as the definition of non-degeneracy on $\Gamma_+$ given in
\cite{BFS}. We show a result (Corollary \ref{caractnodeg}) that helps in the task of testing condition (2) of the previous proposition.
First we recall a result of Kouchnirenko \cite{K} that is stated in the context of Laurent series but that we will state here for germs of $\O_n$.

\begin{thm}\label{caractdelta}\cite[Théorème 6.2]{K} Let $\Delta$ be a compact face of $\Gamma_+$ and let $g_1,\dots, g_s\in \O_n$
such that $\supp(g_i)\subseteq \Gamma_+$, for all $i=1,\dots, s$. Then the following conditions are equivalent:
\begin{enumerate}
\item the ideal of $\RR_\Delta$ generated by ${(g_1)}_\Delta,\dots, {(g_s)}_\Delta$ has finite colength in $\RR_\Delta$;
\item for all compact faces $\Delta'\subseteq \Delta$, the set germ at $0$ of common zeros of ${(g_1)}_{\Delta'},\dots, {(g_s)}_{\Delta'}$
is contained in $\{x\in\C^n: x_1\cdots x_n=0\}$.
\end{enumerate}
\end{thm}

\begin{cor}\label{caractnodeg}
Under the hypothesis of the previous theorem, the following conditions are equivalent:
\begin{enumerate}
\item the ideal of $\RR_\Delta$ generated by $p_\Delta(g_1),\dots, p_\Delta(g_s)$ has finite colength in $\RR_\Delta$;
\item for all compact faces $\Delta'\subseteq \Delta$, the set germ at $0$ of common zeros of $p_{\Delta'}(g_1),\dots, p_{\Delta'}(g_s)$
is contained in $\{x\in\C^n: x_1\cdots x_n=0\}$.
\end{enumerate}
\end{cor}

\begin{proof} Let $r_i=\nu_\Gamma(g_i)$, for all $i=1,\dots, s$, and let $r=r_1\cdots r_s$. Let $I$ denote the ideal of
$\RR_\Delta$ generated by $p_\Delta(g_1),\dots, p_\Delta(g_s)$ and let $J$ denote the ideal of $\RR_\Delta$
generated by $\{p_\Delta(g_1)^{r/r_1},\dots, p_\Delta(g_s)^{r/r_s}\}$.
We observe that $I$ has finite colength in $\RR_\Delta$ if
and only if $J$ has finite colength in $\RR_\Delta$, since $I$ and $J$ have the same radical. It is straightforward to check (see relation (\ref{initial})) that
\begin{equation}\label{pDelta}
p_\Delta(g_i)^{r/r_i}=((g_i)_{\frac{r_i}{M}\Delta})^{r/r_i}=(g_i^{r/r_i})_{\frac{r}{M}\Delta},
\end{equation}
for all $i=1,\dots,s$. Then, by Theorem \ref{caractdelta}, the ideal $J$ has finite colength if and only if
the set germ at $0$ of common zeros of ${(g_1^{r/r_1})}_{A},\dots, {(g_s^{r/r_s})}_{A}$
is contained in $\{x\in\C^n: x_1\cdots x_n=0\}$, for any compact face $A\subseteq {\frac{r}{M}\Delta}$.
Given a subset $\Delta'\subseteq\R^n_+$, we observe that $\Delta'$ is a compact face of $\Delta$ if and only if $\frac{r}{M}\Delta'$ is a compact face of $\frac{r}{M}\Delta$.
Then the result follows immediately from relation (\ref{pDelta}) and Theorem \ref{caractdelta}.
\end{proof}

\begin{cor}\label{conseq} Let $g=(g_1,\dots, g_n):(\C^n,0)\to (\C^n,0)$ be a complex analytic map such that $g^{-1}(0)=\{0\}$.
Then $g$ is non-degenerate on $\Gamma_+$ if and only if the set germ at $0$ of common zeros of $p_{\Delta}(g_1),\dots, p_{\Delta}(g_n)$
is contained in $\{x\in\C^n: x_1\cdots x_n=0\}$, for all compact faces $\Delta$ of $\Gamma_+$.
\end{cor}

\begin{proof} It follows immediately as a consequence of Proposition \ref{revisio} and Corollary \ref{caractnodeg}.
\end{proof}

We will use the following lemma in the proof of the main result (Theorem \ref{main}).

\begin{lem}\label{minsr}
Let $J_1,\dots, J_n$ be ideals of $\O_n$ such that $(J_1,\dots, J_n)$ is non-degenerate on $\Gamma_+$.
Then $(J_1+\A_M^r,\dots, J_n+\A_M^r)$ is also non-degenerate on $\Gamma_+$, for all $r\geq 1$, where $M=M_\Gamma$. That is
\begin{equation}\label{IiJ}
e(J_1+\A_M^r,\dots, J_n+\A_M^r)=
\frac{\min\{r_1, rM\}\cdots\min\{r_n, rM\}}{M^n}n!\V_n(\Gamma_-),
\end{equation}
for all $r\geq 1$, where $r_i=\nu_\Gamma(J_i)$, for all $i=1,\dots,n$.
\end{lem}

\begin{proof}
Let us fix an integer $r\geq 1$. Let $S=\{i: r_i<rM\}$. If $S=\emptyset$, then $J_i\subseteq \A_{rM}$, for all $i=1,\dots, n$. Thus,
since $\overline {\A_M^r}=\overline{\A_{rM}}$ and mixed multiplicities are invariant by integral closures, we have
\begin{align*}
e(J_1+\A_M^r,\dots, J_n+\A_M^r)&=e(J_1+\A_{rM},\dots, J_n+\A_{rM})=e(\A_{rM},\dots, \A_{rM})\\ &=e(\A_{rM})=e(\A_M^r)=r^n n!\V_n(\Gamma_-),
\end{align*}
and the result follows.

Let us suppose that $S\neq\emptyset$. After reordering the integers $r_1,\dots, r_n$, we can assume that $S=\{1,\dots, s\}$, for some $s\geq 1$. Then, we have
$$
e(J_1+\A_M^r,\dots, J_n+\A_M^r)=e(J_1+\A_{rM}, \dots, J_s+\A_{rM}, \A_{rM},\dots, \A_{rM}).
$$

By Proposition \ref{sigmaexists}, there exists an element
$(g_1,\dots, g_n)\in J_1\oplus\cdots \oplus J_n$
such that $\nu_\Gamma(g_i)=r_i$, for all $i=1,\dots, n$, and
\begin{equation}\label{converse}
e(g_1,\dots, g_n)=\sigma(J_1,\dots, J_n)= \frac{r_1\cdots r_n}{M^n}e(\A_M).
\end{equation}

Let $\Delta$ be a compact facet of $\Gamma_+$. Let us denote by $I$ the ideal of
$\RR_\Delta$ generated by $p_\Delta(g_1),\dots, p_\Delta(g_s)$. By Proposition \ref{revisio}, the ideal
of $\RR_\Delta$ generated by $p_\Delta(g_1),\dots, p_\Delta(g_n)$ has finite colength in $\RR_\Delta$. Then,
since $\RR_\Delta$ is Cohen-Macaulay of dimension $n$ (see \cite{Hochster1972} or \cite[p.\ 24]{K}), the ring $\RR_\Delta/I$
is Cohen-Macaulay of dimension $n-s$.

Let $\A_{rM,\Delta}$ be the ideal of $\RR_\Delta$ generated by all monomials $x^k\in\A_{rM}$ such that
$k\in r\Delta$. Let us denote the image of $\A_{rM,\Delta}$ in $\RR_\Delta/I$ by $H$.

Since $\Gamma(\A_{rM,\Delta})=r\Delta$, we have that $\A_{rM,\Delta}$ has finite colength in $\RR_\Delta$, by Theorem \ref{caractdelta}.
Then {\it a fortiori} the ideal $H$ has also finite colength in $\RR_\Delta/I$ and hence $H$ has analytic spread
$n-s$. According to the Northcott-Rees theorem of existence of reductions (see \cite[p.\ 166]{HunekeSwanson2006}), there exist sufficiently general $\C$-linear combinations $h_{s+1},\dots, h_n$ of the set of monomials $\{x^k: k\in r\Delta\}$ such that the ideal generated by the images of $h_{s+1},\dots, h_n$ in $\RR_\Delta/I$ is a reduction of $H$. By the construction of the elements
$h_{s+1},\dots, h_n$, the image of $h_i$ in $\RR_\Delta/I$ equals the image of $(h_i)_{r\Delta}$ in $\RR_\Delta/I$, for all $i=s+1,\dots,n$.
Moreover $(h_i)_{r\Delta}=p_\Delta(h_i)$, for all $i=s+1,\dots, n$.
In particular, the ideal
$\{p_\Delta(g_1),\dots, p_\Delta(g_s), p_\Delta(h_{s+1}),\dots, p_\Delta(h_{n})\}\RR_\Delta$
has finite colength in $\RR_\Delta$.

Since $\Gamma_+$ has a finite number of facets we conclude that there exist $\C$-generic linear combinations $h_{s+1},\dots, h_n$
of $\{x^k: \nu_\Gamma(k)=rM\}$ such that the ideal of $\RR_\Delta$ generated by $p_\Delta(g_1),\dots, p_\Delta(g_s),
p_\Delta(h_{s+1}),\dots, p_\Delta(h_{n})$ has finite colength in $\RR_\Delta$, for all compact facets $\Delta$ of $\Gamma_+$.
In particular, the map $G=(g_1,\dots,g_s, h_{s+1},\dots, h_n):(\C^n,0)\to (\C^n,0)$ is non-degenerate on $\Gamma_+$, by Proposition \ref{revisio} and then
\begin{equation}\label{prel}
e(G)=\frac{r_1\cdots r_s (rM)^{n-s}}{M^n}n!\V_n(\Gamma_-)
=\frac{\min\{r_1, rM\}\cdots\min\{r_n, rM\}}{M^n}n!\V_n(\Gamma_-).
\end{equation}
Moreover we have the following inequalities, as a direct application of Lemma \ref{reverseincl} and Proposition \ref{BezoutlikeNewton}:
\begin{align*}
e(G)=e(g_1,\dots,g_s, h_{s+1},\dots, h_n)&\geq e(J_1+\A_{rM},\dots,J_s+\A_{rM},\A_{rM},\dots,\A_{rM})\\
&\geq \frac{\min\{r_1, rM\}\cdots\min\{r_n, rM\}}{M^n}n!\V_n(\Gamma_-).
\end{align*}
Then the result follows by applying relation (\ref{prel}).
\end{proof}

The following definition is fundamental in our study of \L ojasiewicz exponents via Newton filtrations.

\begin{defn}\label{linkage} Let $J_1,\dots, J_n$ be ideals of $\O_n$ such that $\sigma(J_1,\dots, J_n)<\infty$. Let $r_i=\nu_\Gamma(J_i)$, for $i=1,\dots, n$, let $p=\max\{r_1,\dots, r_n\}$ and let $A=\{i: r_i=p\}$. Let $I$ be a proper ideal of $\O_n$.
We say that the pair $(I; J_1,\dots, J_n)$ is {\it $\Gamma_+$-linked} when there exists some $i_0\in A$ such that
$$
(J_1,\dots, J_{i_0-1},I, J_{i_0+1},\dots, J_n)
$$
is non-degenerate on $\Gamma_+$
\end{defn}

If $g\in\O_n$, then we will write $(I;g,J_2,\dots, J_n)$ instead of $(I;\langle g\rangle, J_2,\dots, J_n)$. We will adopt the same simplification of the notation if any other ideal $J_i$ is generated by only one element, for some $i\in\{1,\dots, n\}$.

Under the conditions of the previous definition, if we assume that $(J_1,\dots, J_n)$ is non-degenerate on $\Gamma_+$, then
$(I; J_1,\dots, J_n)$ is $\Gamma_+$-linked if and only if there exists some $i_0\in A$ such that
\begin{equation}\label{charactlink}
\sigma(J_1,\dots, J_{i_0-1},I, J_{i_0+1},\dots, J_n)=\frac{\nu_\Gamma(I)}{p}\sigma(J_1,\dots, J_n),
\end{equation}
by a direct application of Definition \ref{idealsnodeg}.
In particular, $p$ must be a divisor
of $\nu_\Gamma(I)\sigma(J_1,\dots, J_n)$ in this case (see Example \ref{segonex}).

 Here we show the main result of the article.

\begin{thm}\label{main}
Let $J_1,\dots, J_n$ be a set of ideals of $\O_n$. Let $\nu(J_i)=r_i$, for all $i=1,\dots,n$. Let us
suppose that $(J_1,\dots, J_n)$ is non-degenerate on $\Gamma_+$. Let $I$ be a proper ideal of $\O_n$.
Then
\begin{equation}\label{central}
\LL_I(J_1,\dots, J_n)\leq \LL_I(\B_{r_1},\dots, \B_{r_n}) \leq
\frac{\max\{r_1,\dots,
r_n\}}{\nu_\Gamma(I)}
\end{equation}
and the above inequalities turn into equalities if $(I;J_1,\dots, J_n)$ is
$\Gamma_+$-linked.
\end{thm}

\begin{proof} Along this proof we set $M=M_\Gamma$. By a direct application of Lemma \ref{reverseincl} and Proposition
\ref{BezoutlikeNewton} we have
$$
\sigma(J_1,\dots,J_n)\geq \sigma(\B_{r_1},\dots,\B_{r_n})\geq
\frac{r_1\cdots r_n}{M^n}n!\V_n(\Gamma_+).
$$
Since $(J_1,\dots, J_n)$ is non-degenerate on $\Gamma_+$, the previous inequalities show that
$\sigma(J_1,\dots,J_n)=\sigma(\B_{r_1},\dots,\B_{r_n})$. Hence we can apply Proposition
\ref{uppers} to deduce the inequality
$$
\LL_I(J_1,\dots, J_n)\leq \LL_I(\B_{r_1},\dots, \B_{r_n}).
$$
Let us denote $\max\{r_1,\dots, r_n\}$ and $\nu_\Gamma(I)$ by $p$ and $q$,
respectively. Let us see first that $\LL_I(\B_{r_1},\dots, \B_{r_n})\leq\frac pq$.

Since $\sigma(\B_{r_1},\dots, \B_{r_n})<\infty$, we can compute
the number $r_{\A_M}(\B^s_{r_1},\dots, \B^s_{r_n})$, for all $s\geq 1$:
\begin{align*}
r_{\A_M}&(\B^s_{r_1},\dots, \B^s_{r_n})=\min\left\{r\geq 1:
\sigma(\B^s_{r_1},\dots, \B^s_{r_n})=e(\B^s_{r_1}+\A_M^r,\dots,
\B^s_{r_n}+\A_M^r)\right\}\\
&=\min\left\{r\geq 1: \frac{sr_1\cdots sr_n}{M^n}n!\V_n(\Gamma_-)=\frac{\min\{sr_1, rM\}\cdots \min\{sr_n, rM\}}{M^n}n!\V_n(\Gamma_-)\right\}\\
&=\min\big\{r\geq 1: rM\geq \max\{sr_1,\dots, sr_n\}\big\}\\
&=\min\left\{r\geq 1: r\geq \frac{\max\{sr_1,\dots,
sr_n\}}{M}\right\}
=\left\lceil\frac{\max\{sr_1,\dots, sr_n\}}{M}\right\rceil,
\end{align*}
where $\lceil a\rceil$ denotes the least integer greater than or
equal to $a$, for any $a\in \R$, and the second equality is a direct
application of Lemma \ref{minsr}.
Therefore
\begin{align*}
\LL_{\A_M}(\B_{r_1},\dots, \B_{r_n})&=
\inf_{s\geq 1}\frac{r_{\A_M}(\B^s_{r_1},\dots, \B^s_{r_n})}{s}
\leq
\inf_{a\geq 1}\frac{r_{\A_M}(\B^{aM}_{r_1},\dots,\B^{aM}_{r_n})}{aM}\\
&=\inf_{a\geq 1}\frac{1}{aM}
\left\lceil\frac{\max\{aM r_1,\dots,aM r_n\}}{M}\right\rceil
=\frac{\max\{r_1,\dots, r_n\}}{M}.
\end{align*}

By \cite[Théorème 6.3]{LT2008} we have the relation $\LL_I(\A_M)=\frac{1}{\overline{\nu}_{\A_M}(I)}$,
where $\overline{\nu}_{\A_M}$ is the asymptotic Samuel function of $\A_M$. We observe that $\Gamma_+(\A_M)=\Gamma_+$.
Then, since $\A_M$ is a monomial ideal we have
$$
\LL_I(\A_M)=\frac{M}{\nu_\Gamma(I)},
$$
as a consequence of Proposition \ref{Impa}. Therefore, by Lemma \ref{transit} we obtain
\begin{align*}
\LL_I(\B_{r_1},\dots, \B_{r_n})&\leq
\LL_I(\A_M)\LL_{\A_M}(\B_{r_1},\dots,\B_{r_n}) \\
&\leq\frac{M}{\nu_\Gamma(I)}\frac{\max\{r_1,\dots,
r_n\}}{M}
=\frac{\max\{r_1,\dots, r_n\}}{\nu_\Gamma(I)}=\frac{p}{q}.
\end{align*}

Supposing that $(I; J_1,\dots, J_n)$ is $\Gamma_+$-linked, let us prove that $\LL_I(J_{1},\dots, J_{n})\geq \frac{p}{q}$.
By the definition of $\LL_I(J_{1},\dots, J_{n})$, this inequality holds if
and only if
$$
\frac{r_I(J_1^s,\dots, J_n^s)}{s}\geq \frac pq
$$
for all $s\geq 1$. By Lemma \ref{rpowers} we have that
$qr_I(J_1^s,\dots, J_n^s)\geq r_I(J_1^{sq},\dots, J_n^{sq})$, for all
$s\geq 1$. Therefore it suffices to show that
\begin{equation}\label{suficient1}
r_I(J_1^{sq},\dots, J_n^{sq})> sp-1,
\end{equation}
for all $s\geq 1$. Let us fix an integer $s\geq 1$. Then relation
(\ref{suficient1}) is equivalent to saying that
\begin{equation}\label{suficient2}
\sigma(J_1^{sq},\dots, J_n^{sq})> \sigma(J_1^{sq}+I^{sp-1}, \dots,
J_n^{sq}+I^{sp-1}).
\end{equation}

By \cite[Lemma 2.6]{Bivia2009} we have
\begin{equation}\label{primera}
\sigma(J^{sq}_1,\dots, J^{sq}_n)=
(sq)^n\sigma(J_1,\dots, J_n).
\end{equation}
Let us denote the multiplicity $\sigma(J_1,\dots, J_n)$ also by $\sigma$.

Let $A=\{i: r_i=p\}$. By hypothesis, there exists an index $i_0\in A$ such that
$$
(J_1,\dots, J_{i_0-1},I,J_{i_0+1},\dots, J_n)
$$
is non-degenerate on $\Gamma_+$.
Then, we have $\sigma(J_1,\dots, J_{i_0-1},I,J_{i_0+1},\dots, J_n)=\frac qp\sigma$ (see relation (\ref{charactlink})) and hence
\begin{align}
(sq)^{n-1}(sp-1)\,\frac{q}{p}\sigma &=\sigma(J_{1}^{sq},\dots, J_{i_0-1}^{sq}, I^{sp-1},J_{i_0+1}^{sq}\dots, J_n^{sq})\nonumber   \\ &\geq \sigma(J_1^{sq}+I^{sp-1}, \dots, J_n^{sq}+I^{sp-1})\label{segona},
\end{align}
where the second inequality comes from Lemma \ref{reverseincl}. An elementary computation shows that
$$
(sq)^n\sigma>(sq)^{n-1}(sp-1)\,\frac{q}{p}\sigma
$$
if and only if $sp>sp-1$, which is the case. Thus, comparing (\ref{primera}) and (\ref{segona})
we conclude that relation (\ref{suficient2}) holds and hence the result is proven.
\end{proof}

\begin{ex}\label{primerex} Let us consider the ideals $J_1$ and $J_2$ of $\O_2$ given by $J_1=\langle x^5, x^2y^2, y^5 \rangle$
and $J_2=\langle x^3y^3\rangle$. Let $\Gamma_+=\Gamma_+(J_1)$. The filtrating map associated to $\Gamma_+$ is given by
$\phi_\Gamma(k_1,k_2)=\min\{2k_1+3k_2,3k_1+2k_2\}$, for all $(k_1,k_2)\in\R^2_+$. Hence $\nu_\Gamma(J_1)=M_\Gamma=10$,
$\nu_\Gamma(J_2)=15$ and $\nu_\Gamma(m)=2$. Using the program {\it Singular} \cite{Singular} we check that
$\sigma(J_1,J_2)=30$. Then we have the relation
$$
\sigma(J_1,J_2)=30=\frac{\nu_\Gamma(J_1)\nu_\Gamma(J_2)}{M_\Gamma^2}e(J_1),
$$
which shows that $(J_1,J_2)$ is non-degenerate on $\Gamma_+$. Moreover we have that relation (\ref{charactlink}) holds in this context, that is
$$
\sigma(J_1, m)=4=\frac{\nu_\Gamma(m)}{\max\{\nu_\Gamma(J_1),\nu_\Gamma(J_2)\}}\sigma(J_1,J_2).
$$
Then $\LL_0(J_1,J_2)=\frac{15}{2}$, by Theorem \ref{main}.
\end{ex}

\begin{ex}\label{segonex} Let us consider the ideal of $\O_2$ given by $J=\langle x^4, xy, y^5\rangle$
and let $\Gamma_+=\Gamma_+(J)$. We observe that $e(J)=9$ and that the filtrating map associated to $\Gamma_+$ is given by
$\phi_\Gamma(k_1,k_2)=20\min\{\frac{k_1+3k_2}{4},\frac{4k_1+k_2}{5}\}$, for all $(k_1,k_2)\in\R^2_+$, with $M_\Gamma=20$.
Let $J_1=\langle x^2y^2, x^8\rangle$ and let $J_2=\langle xy, y^5\rangle$.
Then we observe that $\nu_\Gamma(J_1)=40$, $\nu_\Gamma(J_2)=20$ and moreover
$$
\sigma(J_1,J_2)=18=\frac{\nu_\Gamma(J_1)\nu_\Gamma(J_2)}{20^2}e(J).
$$
Then $(J_1,J_2)$ is non-degenerate on $\Gamma_+$. If $(m;J_1,J_2)$ were $\Gamma_+$-linked
then the multiplicity $\sigma(m,J_2)$ would be equal to $\frac{\nu_\Gamma(m)\nu_\Gamma(J_2)}{20^2}e(J)=\frac{9}{5}\notin\Z$.
However the \L ojasiewicz exponent $\LL_0(J_1,J_2)$ attains the maximum possible value, that is
$$
\LL_0(J_1,J_2)=10=\frac{\max\{\nu_\Gamma(J_1), \nu_\Gamma(J_2)\}}{\nu_\Gamma(m)}=\frac{40}{4}=10,
$$
where the first equality follows from \cite[\S 4]{Bivia2009} (see also \cite{Lenarcik}).
\end{ex}

\begin{ex}\label{tercerex} Let us fix an integer $a\geq 2$. Let us consider the map $g=(g_1,g_2,g_3):(\C^3,0)\to (\C^3,0)$ given by
\begin{align*}
g_1(x,y,z)&=x^6+y^6-z^5+xyz\\
g_2(x,y,z)&=x^2y^2z^2\\
g_3(x,y,z)&=y^{6a}+z^{5a}.
\end{align*}
Let $\Gamma_+=\Gamma_+(g_1)$. We have that $M_\Gamma=30$ and the Newton filtration $\nu_\Gamma$ is determined by the filtrating map
$\phi_\Gamma:\R^3_+\to \R_+$ given by
$$
\phi_\Gamma(k_1,k_2,k_3)=\min\big\{19k_1+5k_2+6k_3,\, 5k_1+19k_2+6k_3,\, 5\big(k_1+k_2+4k_3\big)\big\}.
$$
Using Corollary \ref{conseq}, it is immediate to check that
the map $g$ is non-degenerate on $\Gamma_+$. Moreover $\nu_\Gamma(g_1)=30$, $\nu_\Gamma(g_2)=60$ and $\nu_\Gamma(g_3)=30a$.

Let $m$ denote the maximal ideal of $\O_3$. Let $h$ denote a $\C$-generic linear form. Then $\nu_\Gamma(m)=\nu_\Gamma(h)=5$.
Again by Corollary \ref{conseq}, we obtain that the map $(g_1,g_2,h)$ is non-degenerate on $\Gamma_+$. Hence we conclude that $(m; g_1,g_2,g_3)$ is $\Gamma_+$-linked. Then, by Theorem \ref{main}
we obtain
$$
\LL_0(g_1, g_2, g_3)=\frac{\max\{\nu_\Gamma(g_1), \nu_\Gamma(g_2), \nu_\Gamma(g_3)\}}{\nu_\Gamma(m)}=6a.
$$
\end{ex}

\begin{defn} Let $\Delta$ be a compact facet of $\Gamma_+$. We say that $\Delta$ is an {\it inner} facet of $\Gamma_+$ when no vertex of $\Delta$ is contained in some coordinate axis of $\R^n$.
\end{defn}

\begin{rem}\label{starshaped}
If $g=(g_1,\dots, g_n):(\C^n,0)\to (\C^n,0)$ is a non-degenerate map on $\Gamma_+$
and $\Delta$ denotes a compact facet of $\Gamma_+$, then the ideal of $\RR_\Delta$ generated by
$p_\Delta(g_1),\dots, p_\Delta(g_n)$ has finite colength in $\RR_\Delta$ (see Proposition \ref{revisio}). In particular
$p_\Delta(g_i)\neq 0$, for all $i=1,\dots, n$, since $\dim (\RR_\Delta)=n$.

We observe that if $h$ is a $\C$-linear form and $\Delta$ is an inner facet of $\Gamma_+$ then $p_\Delta(h)=0$.
Then, if we suppose that the pair $(m; g_1,\dots, g_n)$ is $\Gamma_+$-linked, we are forcing the Newton polyhedron
$\Gamma_+$ to have no inner facets. The same happens when replacing the maximal ideal $m$ by any ideal
whose support is contained in the union of the coordinate axis. Therefore, in the aim of applying Theorem \ref{main} to obtain the exact value
of $\LL_0(g_1,\dots, g_n)$ via the Newton filtration induced by $\Gamma_+$, we need the Newton polyhedron $\Gamma_+$ to have no inner facets.

We remark that the Newton polyhedra $\Gamma_+$ that appear in Examples \ref{primerex} and \ref{tercerex} do not have inner facets.
\end{rem}

%%%%%%%%%%%%%%%%%%%%%%%%%%%%%%%%%%%%%%%%%%%%%%%%%%%%%%%%%%%%%%%%%%%%%%%%%%%%%%%%%%%%%%%%%%%%%%%%%%%%%%%%%%%%%%%%%%
\section{Applications to weighted homogeneous filtrations}\label{whfiltrations}

Along this section we will denote by $w$ a primitive vector $w=(w_1,\dots,w_n)\in\Z^n_{\geq 1}$. Let $w_0=\min_iw_i$ and let $A_w=\{i: w_i=w_0\}$.
If $h\in\O_n$, $h\neq0$, then we denote by $d_w(h)$ the $w$-degree of $h$, that is, $d_w(h)=\min\{\langle w,k\rangle: k\in\supp(h)\}$, where $\langle\,, \rangle$ denotes the standard scalar product in $\R^n$. If $h=0$ then we set $d_w(h)=+\infty$. Moreover, if $J$ is an ideal of $\O_n$ then we
define $d_w(J)=\min\{d_w(h): h\in J\}$.

Let us denote by $\Gamma_+^w$ the Newton polyhedron in $\R^n_+$ determined by
$\{\frac{w_1\cdots w_n}{w_1}e_1,\dots, \frac{w_1\cdots w_n}{w_n}e_n\}$ and by $\Gamma^w$ the Newton boundary of $\Gamma_+^w$. It is straightforward to see that $\Gamma^w$ has only one compact facet, which
is supported by $w$, and that the weighted homogeneous filtration induced by $w$ (see \cite[Section 4]{BiviaEncinas2011})
coincides with the Newton filtration of $\O_n$ induced by $\Gamma_+^w$. That is $d_w(h)=\nu_{\Gamma^w}(h)$, for all $h\in\O_n$.

In \cite[p.\ 584]{BiviaEncinas2011} we introduced the following definition, which lead us to formulate
a sufficient condition for $\LL_0(J_1,\dots, J_n)$ to attain the bound $\frac{1}{w_0}\max\{d_w(J_1),\dots, d_w(J_n)\}$ (see \cite[Theorem 4.7]{BiviaEncinas2011} for details) and to derive interesting consequences (see \cite[Proposition 4.14 and Corollary 4.16]{BiviaEncinas2011}).

\begin{defn}\label{defwmatching}\cite[p.\ 584]{BiviaEncinas2011}
Let $J_1,\dots, J_n$ be a family of ideals of $\O_n$ and let
$r_i=d_w(J_i)$, for all $i=1,\dots, n$. We say that the $n$-tuple of ideals $(J_1,\dots, J_n)$
\textit{admits a $w$-matching} if there exist a permutation $\tau$
of $\{1,\ldots,n\}$ and an index $i_{0}\in\{1,\ldots,n\}$ such that
\begin{enumerate}
    \item[(a)]  $w_{i_{0}}=\min\{w_{1},\ldots,w_{n}\}$,
    \item[(b)]  $r_{\tau(i_{0})}=\max\{r_{1},\ldots,r_{n}\}$ and
    \item[(c)]  the pure monomial $x_{i}^{r_{\tau(i)}/w_i}$
    belongs to $J_{\tau(i)}$, for
    all $i\neq i_{0}$.
\end{enumerate}
\end{defn}

\begin{rem}\label{puremon}
Under the conditions of the above definition, if $(J_1,\dots, J_n)$ admits a $w$-matching, then
$J_i$ contains some pure monomial, for all $i$ such that $r_i<\max\{r_1,\dots, r_n\}$.
\end{rem}

Let $J_1,\dots, J_n$ be ideals of $\O_n$ and let $r_i=d_w(J_i)$, for all $i=1,\dots, n$.
In order to simplify the nomenclature, we will say that $(J_1,\dots, J_n)$ is {\it $w$-non-degenerate}
when $(J_1,\dots, J_n)$ is non-degenerate on $\Gamma_+^w$, which is to say that
$\sigma(J_1,\dots, J_n)=\frac{r_1\cdots r_n}{w_1\cdots w_n}$ (see Definition \ref{idealsnodeg}).
Moreover, we will say that the pair $(m; J_1,\dots, J_n)$ {\it $w$-linked} when it is {\it $\Gamma_+^w$-linked} (see Definition \ref{linkage}).

Let $p=\max\{r_1,\dots, r_n\}$ and let $A=\{i: r_i=p\}$.
Then it follows that
the pair $(m; J_1,\dots, J_n)$ is $w$-linked if and only if there exists some $i_0\in A$ such that
\begin{align}
\sigma(J_1,\dots, J_{i_0-1},m, J_{i_0+1},\dots, J_n)=
\frac{\min\{w_1,\dots,w_n\}}{\max\{r_1,\dots, r_n\}}\frac{r_1\cdots r_n}{w_1\cdots w_n},
\end{align}
as a direct application of Definition \ref{linkage}.

\begin{lem}\label{casparticular} Let $J_1,\dots, J_n$ be a family of ideals of $\O_n$. Let us suppose that
$J_1,\dots, J_n$ admits a $w$-matching. Then $(m; J_1,\dots, J_n)$ is {\it $w$-linked}.
\end{lem}

\begin{proof} Let $r_i=d_w(J_i)$, for all $i=1,\dots, n$. Let $i_0\in\{1,\dots, n\}$ verifying conditions (a), (b) and (c)
of Definition \ref{defwmatching}, for some permutation $\tau$ of $\{1,\dots, n\}$. We observe that
\begin{align}
\frac{r_1\cdots r_n}{w_1\cdots w_n}\frac{w_{i_0}}{r_{\tau(i_0)}} &= e\left(x_1^{r_{\tau(1)}/w_1},\dots, x_{i_0-1}^{r_{\tau(i_0-1)}/w_{i_0-1}},x_{i_0},
x_{i_0+1}^{r_{\tau(i_0+1)}/w_{i_0+1}},\dots, x_{n}^{r_{\tau(n)}/w_{n}} \right)\nonumber\\
&\geq e\left(J_{\tau(1)},\dots, J_{\tau(i_0-1)}, m, J_{\tau(i_0+1)},\dots, J_{\tau(n)}\right) \geq \frac{r_1\cdots r_n}{w_1\cdots w_n}\frac{w_{i_0}}{r_{\tau(i_0)}}\label{second},
\end{align}
where the first inequality of (\ref{second}) is a direct consequence of condition (c) together
with Lemma \ref{reverseincl}, and the second inequality of (\ref{second})
follows immediately from Proposition \ref{BezoutlikeNewton}. Therefore these inequalities are actually equalities. That is,
if we denote $\tau(i_0)$ by $j_0$, then
$$
\sigma(J_1,\dots,J_{j_0-1},m,J_{j_0+1},\dots, J_n)=\frac{r_1\cdots r_n}{w_1\cdots w_n}\frac{w_{i_0}}{r_{j_0}},
$$
which is to say that $(J_1,\dots,J_{j_0-1},m,J_{j_0+1},\dots, J_n)$ is $w$-non-degenerate.
\end{proof}

The converse of Lemma \ref{casparticular} is not true in general, as the following simple example shows.

\begin{ex} Let us consider the ideals $J_1$ and $J_2$ of $\O_2$ given by $J_1=\langle xy\rangle$,
$J_2=m_2^4$ and let $w=(1,1)$. We observe that $\sigma(J_1,m_2)=2=d_w(J_1)d_w(m_2)$. Then $(J_1,J_2)$ is $w$-linked. On the other hand
$(J_1,J_2)$ does not admit a $w$-matching, since $J_1$ does not have any pure monomial (see Remark \ref{puremon}).
\end{ex}

\begin{ex} Let $w=(1,1,2)$ and let us consider the ideals of $\O_3$ given by
$J_1=\langle xy\rangle$ and $J_2=\langle x^4,z^2\rangle$. We observe that $d_w(J_1)=2$, $d_w(J_2)=4$ and
\begin{equation}\label{sigmaquatre}
\sigma(J_1,J_2,m)=4=\frac{d_w(J_1)d_w(J_2)d_w(m)}{2}.
\end{equation}
Let $J_3$ denote any ideal of $\O_3$ of $w$-degree strictly greater than $4$ such that $(J_1,J_2,J_3)$ is $w$-non-degenerate.
Then relation (\ref{sigmaquatre}) implies that $(m;J_1,J_2,J_3)$ is $w$-linked and hence $\LL_0(J_1,J_2,J_3)=d_w(J_3)$, by Theorem \ref{main}.
We observe that $(J_1,J_2,J_3)$ does not admit a $w$-matching, since $J_1$ does not contain any pure monomial.
\end{ex}

%---------------------------------------------------------------------------------------------------------------------

If $f\in\O_n$ and $w\in\Z^n_{\geq 1}$ then we denote by $p_w(f)$
the {\it principal part of $f$ with respect to $w$}. That is, if $f=\sum_ka_kx^k$ is the Taylor
expansion of $f$ around the origin, then $p_w(f)$ is the sum of all terms $a_kx^k$ such that
$\langle w,k\rangle=d_w(f)$. We say that $f$ is {\it weighted homogeneous with respect to $w$} when $p_w(f)=f$.
Moreover, the function $f$ is termed {\it semi-weighted homogeneous with respect to $w$} when $p_w(f)$ has an isolated
singularity at the origin.

Let $g=(g_1,\dots, g_n):(\C^n,0)\to (\C^n,0)$ be an analytic map germ.
We say that $g$ is {\it weighted homogeneous with respect to $w$ of degree $(r_1,\dots, r_n)$}
when $g_i$ is weighted homogeneous with respect to $w$ of degree $r_i$, for all $i=1,\dots, n$.
Moreover, we define the map $p_w(g):(\C^n,0)\to (\C^n,0)$ by $p_w(g)=(p_w(g_1),\dots, p_w(g_n))$ and we also define $d_w(g)=(d_w(g_1),\dots, d_w(g_n))$.

If $g^{-1}(0)=\{0\}$, we recall that we denote by $e(g_1,\dots, g_n)$, or by $e(g)$, the multiplicity of the ideal of $\O_n$
generated by $g_1,\dots, g_n$.

Given a vector $r=(r_1,\dots, r_n)\in\Z^n_{\geq 1}$, then we denote by $H(w;r)$ the set of all weighted homogeneous
maps $g:(\C^n,0)\to (\C^n,0)$ for which $d_w(g)=(r_1,\dots, r_n)$ and $g^{-1}(0)=\{0\}$.

We say that $g$ is {\it semi-weighted homogeneous with respect to $w$} when $p_w(g)^{-1}(0)=\{0\}$. In this case we have
$e(g)=e(p_w(g))=\frac{r_1\cdots r_n}{w_1\cdots w_n}$, where
$r_i=d_w(g_i)$, for $i=1,\dots, n$ (see \cite[Section 12]{ArnoldGuseinVarchenko1985}).
It is straightforward to see that if $f\in\O_n$ then $f$
is semi-weighted homogeneous with respect to $w$ if an only if its gradient map $\nabla f=(\frac{\partial f}{\partial x_1},\dots,
\frac{\partial f}{\partial x_n}):(\C^n,0)\to (\C^n,0)$ is semi-weighted homogeneous with respect to $w$.

Let $g=(g_1,\dots, g_n):(\C^n,0)\to (\C^n,0)$ be a complex analytic map. In order to simplify the nomenclature,
we say that $g$ is {\it $w$-linked} when the pair $(m; g_1,\dots, g_n)$ is $w$-linked.
It follows from Definition \ref{linkage} and Proposition \ref{revisio} that $g$ is $w$-linked if and only if $p_w(g)$ is $w$-linked.

%---------------------------------------------------------------------------------------------------------------------

If $J$ is an ideal of $\O_n$ then we denote by $p_w(J)$ the ideal of $\O_n$ generated by
$\{p_w(h): h\in J,\, d_w(h)=d_w(J)\}$. If $J$ is a monomial ideal then we have $p_w(J)\subseteq J$.
The following result helps in the task of checking the condition of $w$-linkage.

\begin{lem}\label{oneface} Let $J_1,\dots, J_n$ be ideals of $\O_n$ such that $\sigma(J_1,\dots, J_n)<\infty$. Then
$(J_1,\dots, J_n)$ is $w$-non-degenerate if and only if $\sigma\left(p_w(J_1), \dots,p_w(J_n)\right)<\infty$.
\end{lem}

\begin{proof}
By Proposition \ref{sigmaexists}, we can consider a sufficiently general element $(g_1,\dots, g_n)$ of
$J_1\oplus \cdots \oplus J_n$ such that $\sigma(J_1,\dots, J_n)=e(g_1,\dots,g_n)$.
Then $(J_1,\dots, J_n)$ is $w$-non-degenerate if and only if $(g_1,\dots,g_n)$ is also, which is to say that
for any compact facet $\Delta$ of $\Gamma_+^w$, the ideal of $\RR_\Delta$ generated by
$p_\Delta(g_1),\dots, p_\Delta(g_n)$ has finite colength in $\RR_\Delta$, by Corollaries \ref{caractnodeg} and \ref{conseq}.
Moreover $\Gamma_+^w$ has only one compact facet $\Delta$, $\RR_\Delta=\O_n$ and $p_\Delta(g_i)=p_w(g_i)$, for all
$i=1,\dots, n$. Then the result follows.
\end{proof}

\begin{rem}\label{simpl}
Under the conditions of the above result, let $p=\max\{d_w(J_1),\dots, d_w(J_n)\}$. Let $m_w=\langle x_i: i\in A_w\rangle$. Then, as a consequence of Lemma \ref{oneface}, have that $(m;J_1,\dots,J_n)$ is $w$-linked if and only if there exists some $i_0\in \{1,\dots, n\}$ for which $d_w(J_{i_0})=p$ and
$$
\sigma\big(p_w(J_{1}), \dots, p_w(J_{i_0-1}),m_w,p_w(J_{i_0+1}),\dots, p_w(J_{n})\big)<\infty.
$$

In particular, if $A_w=\{i_0\}$, for some $i_0\in\{1,\dots, n\}$, and $f:(\C^n,0)\to (\C,0)$ is a germ of analytic function with an isolated singularity at the origin, then $\nabla f$ is $w$-linked if and only if
$$
\sigma\left( \frac{\partial p_w(f)}{\partial x_1},\dots, \frac{\partial p_w(f)}{\partial x_{i_0-1}},
x_{i_0}, \frac{\partial p_w(f)}{\partial x_{i_0+1}}, \dots,\frac{\partial p_w(f)}{\partial x_n}\right)<\infty.
$$
\end{rem}

\begin{cor}\label{corg} Let $g=(g_1,\dots, g_n):(\C^n,0)\to (\C^n,0)$ be a semi-weighted homogeneous map with respect to $w$. Let $r_i=d_w(g_i)$, for all $i=1,\dots, n$. Then
\begin{equation}\label{cadena}
\LL_0(p_w(g))\leq \LL_0(g)\leq \frac{\max\{r_1,\dots, r_n\}}{\min\{w_1,\dots, w_n\}}
\end{equation}
and both inequalities turn into equalities if $g$ is $w$-linked. In particular
if $f:(\C^n,0)\to (\C,0)$ is a semi-weighted homogeneous function of degree $d$ with respect to $w$, then
\begin{equation}%\label{Lojf}
\LL_0(\nabla f)\leq \frac{d-w_0}{w_0},
\end{equation}
and equality holds if $\nabla f$ is $w$-linked.
\end{cor}

\begin{proof} Let us see the first inequality of (\ref{cadena}). Let $I_i=\langle x^k: k\in\Gamma_+(p_w(g_i))\rangle$ and
let $J_i=\langle x^k: k\in\Gamma_+(g_i)\rangle$, for all $i=1,\dots, n$.
Since $\Gamma_+(p_w(g_i))\subseteq \Gamma_+(g_i)$, we have that $I_i\subseteq J_i$, for all $i=1,\dots,n$. Thus
\begin{equation}\label{epwg}
e(p_w(g))\geq \sigma(I_1,\dots, I_n)\geq \sigma(J_1,\dots, J_n)\geq \frac{r_1\cdots r_n}{w_1\cdots w_n},
\end{equation}
by Lemmas \ref{reverseincl} and \ref{BezoutlikeNewton}.
We know that $e(p_w(g))=e(g)=\frac{r_1\cdots r_n}{w_1\cdots w_n}$ (see \cite[Section 12]{ArnoldGuseinVarchenko1985}). Hence $\sigma(I_1,\dots, I_n)=\sigma(J_1,\dots, J_n)$, by relation (\ref{epwg}). Thus $\LL_0(I_1,\dots, I_n)\leq\LL_0(J_1,\dots, J_n)$, by Proposition \ref{uppers}.
From Theorem \ref{base} we have $\LL_0(I_1,\dots, I_n)=\LL_0(p_w(g))$ and $\LL_0(J_1,\dots, J_n)=\LL_0(g)$.
Then $\LL_0(p_w(g))\leq \LL_0(g)$. This inequality also follows as a consequence of the lower semi-continuity
of \L ojasiewicz exponents in deformations with constant multiplicity, as is proven by P\l oski in \cite{PloskiArXiv}.

The second inequality of (\ref{cadena}) and the
equality under the hypothesis that $g$ is $w$-linked are direct consequences of Theorem \ref{main}.
\end{proof}

Let us consider the case $w=(1,\dots,1)$.
Let $f$ be a semi-weighted homogeneous function of degree $d$ with respect to
$w$ and let $h=p_w(f)$. In particular, $h$ is a homogeneous polynomial.
Since $h$ has an isolated singularity at the origin and $m_w=m$ in this case,
we have clearly that $\nabla f$ is $w$-linked. Hence we obtain the well-known equality $\LL_0(\nabla f)=d-1$
(see for instance \cite[Proposition 2.2]{Ploski1985} or \cite[Theorem 3.5]{Ploski1988}).

The inequality $\LL_0(p_w(g))\leq \LL_0(g)$ of (\ref{cadena}) can be strict, as the following example shows.

\begin{ex}\label{pw} Let us consider the maps $g,g':(\C^2,0)\to (\C^2,0)$ given by
\begin{align*}
g(x,y)&=(x^3,y^4+xy)\\
g'(x,y)&=(x^3,y^4).
\end{align*}
By an application of Proposition \ref{Impa} we have $\LL_0(g')=4$. Moreover $\LL_0(g)=9$, by applying the main result of \cite{Lenarcik} about
the computation of \L ojasiewicz exponents in dimension 2 (or \cite[Section 4]{Bivia2009}). Then we see that
\begin{equation}\label{exemplet}
\LL_0(g')=4<9=\LL_0(g).
\end{equation}

Let $w=(w_1,w_2)\in\Z^2_{\geq 1}$ be any vector such that $w_1>3w_2$. We observe that $g$ is semi-weighted homogeneous
with respect to $w$ with $p_w(g)=g'$. Using Remark \ref{simpl} we observe that
$g$ is not $w$-linked with respect to such vectors of weights.

Moreover, we have that $g,g'\in H((3,1); (9,4))$. Then, relation (\ref{exemplet}) also shows that
if $g,g'\in H(w;r)$, then $g$ and $g'$ do not have the same \L ojasiewicz exponent in general.
\end{ex}

\begin{rem}
We do not know if an example similar to the previous one exists for gradient maps. That is, if we fix a vector
of weights $w\in\Z^n_{\geq 1}$ and $f,f'$ are weighted-homogeneous functions with respect to $w$ with the same
$w$-degree and with an isolated singularity at the origin, it is still not known in general if $\LL_0(\nabla f)=\LL_0(\nabla f')$.
Obviously, this equality holds when $\nabla f$ is $w$-linked and $\nabla f'$ is also, by virtue of Corollary \ref{corg}.

On the other hand, if $f$ is a semi-weighted homogeneous function with respect to $w$, it is still an open question to determine if the equality
$\LL_0(\nabla f)=\LL_0(\nabla p_w(f))$ holds. We recall that $f$ and $p_w(f)$ are topologically equivalent
(see for instance \cite[Corollary 5]{DamonGaffney} or \cite[Corollary 2.1]{Parusinski} for a more general result).
Then, if $f:(\C^n,0)\to (\C,0)$ denotes an arbitrary analytic function germ
with an isolated singularity at the origin we remark that it is not known in general if 
$\LL_0(\nabla f)=\LL_0(\nabla (f\circ\phi))$, where $\phi:(\C^n,0)\to (\C^n,0)$
denotes a germ of homeomorphism.
%the \L ojasiewicz exponent $\LL_0(\nabla f)$ is an invariant
%under changes of coordinates of $f$ that are homeomorphisms.
\end{rem}

If $g:(\C^n,0)\to (\C^n,0)$ is a complex analytic map, then we denote by $I(g)$ the ideal of $\O_n$ generated
by the component functions of $g$. The next result shows a sufficient condition for equality in the first
inequality of (\ref{cadena}).

\begin{prop}\label{Ig} Let $g=(g_1,\dots, g_n):(\C^n,0)\to (\C^n,0)$ be a semi-weighted homogeneous map with respect to $w$.
Let $h_1=p_w(g)$ and let $h_2=g-h_1$. If $I(h_2)\subseteq \overline{m\cdot I(h_1)}$, then
$$
\LL_0(p_w(g))=\LL_0(g).
$$
\end{prop}

\begin{proof}
By Corollary \ref{corg} we have $\LL_0(p_w(g))\leq \LL_0(g)$. Let us see the reverse inequality.
Using the hypothesis $I(h_2)\subseteq \overline{m\cdot I(h_1)}$ we observe that
$$
I(h_1)\subseteq I(g)+I(h_2)\subseteq I(g)+\overline{m\cdot I(h_1)}\subseteq \overline{I(g)+m\cdot I(h_1)}.
$$
Then, by the integral Nakayama Lemma (see \cite[p.\ 324]{Teissier1973}) we obtain the inclusion
$$
I(h_1)\subseteq \overline{I(g)},
$$
which implies that $\LL_0(I(h_1))\geq \LL_0(I(g))$, that is $\LL_0(p_w(g))\geq \LL_0(g)$ and hence the result follows.
\end{proof}

\begin{cor}\label{delfinal} Let $J_1,\dots, J_n$ be ideals of $\O_n$ such that $(J_1,\dots, J_n)$ is $w$-non-degenerate.
Let $r_i=d_w(J_i)$, for all $i=1,\dots, n$. If there exists some $i_0\in\{1,\dots, n\}$
such that $r_{i_0}=\max\{r_1,\dots r_n\}$ and $p_w(J_{i_0})\subseteq p_w(I)$, then
$$
\LL_I(J_1,\dots, J_n)\leq \frac{\max\{r_1,\dots, r_n\}}{d_w(I)}.
$$
\end{cor}

\begin{proof}
Since $(J_1,\dots, J_n)$ is $w$-non-degenerate we have $\sigma\left(p_w(J_1), \dots,p_w(J_n)\right)<\infty$, by Lemma \ref{oneface}.
Thus the inclusion $p_w(J_{i_0})\subseteq p_w(I)$ leads to obtain
$$
\infty>\sigma\left(p_w(J_1), \dots,p_w(J_n)\right)\geq \sigma\left(p_w(J_1), \dots,p_w(J_{i_0-1}), p_w(I),p_w(J_{i_0+1}),\dots, p_w(J_n)\right),
$$
by Lemma \ref{reverseincl}. Then the pair $(I; J_1,\dots, J_n)$ is $w$-linked by Lemma \ref{oneface} and thus the result
follows by Theorem \ref{main}.
\end{proof}

Given a function $f\in\O_n$ and an integer $i\in\{1,\dots, n\}$, we denote by $\supp_i(f)$ the set of those $k\in\supp(f)$
such that $k_i\neq 0$. It is clear that if $\supp_i(f)=\emptyset$, for some $i\in\{1,\dots, n\}$, then $f$ does not depend on the variable $x_i$ and therefore $f$ does not have an isolated singularity at the origin. The next result, which is an immediate consequence of Corollary \ref{delfinal}, allows to determine easily the \L ojasiewicz exponent $\LL_0(\nabla f)$ for an ample class of functions $f\in \O_n$.

\begin{cor}\label{sufficient}
Let $f:(\C^n,0)\to (\C,0)$ be a semi-weighted homogeneous function of degree $d$ with respect to $w$.
Let $h=p_w(f)$. Let us suppose that there exists some $i_0\in A_w$ such that
$\frac{\partial f}{\partial x_{i_0}}\in m_w$. Then
\begin{equation}\label{eqL}
\LL_0(\nabla h)=\LL_0(\nabla f)=\frac{d-w_0}{w_0}.
\end{equation}

In particular, if $A_w=\{i_0\}$, for some $i_0\in\{1,\dots, n\}$, and $k_{i_0}\geq 2$, for all $k\in\supp_{i_0}(f)$, then
the above equality holds.
\end{cor}

%\begin{ex} Let us consider the function $f:(\C^2,0)\to (\C,0)$ given by $f(x,y)=xy+x^3$. We observe that
%$f$ is weighted homogeneous of degree $d=3$ with respect of $w=(1,2)$ and it is clear that $\LL_0(\nabla f)=1$, which is different
%from $\frac{d-w_0}{w_0}=2$. Moreover, we see that $f$ does not satisfy the hypothesis of Theorem \ref{sufficient}.
%\end{ex}

%\begin{ex} Let us consider the function $f:(\C^4,0)\to (\C,0)$ given by
%$$
%f(x,y,z,t)=x^{10}+xz^3+y^5+y^2z^2+t^2
%$$
%This function is weighted homogeneous of degree $d=10$ with respect to $w=(1,2,3,5)$ and has an isolated singularity at the origin.
%On the other hand, this function is not $w$-linked. However $\LL_0(\nabla f)$ attains the maximum
%possible value, that is $\LL_0(\nabla f)=\frac{d-w_0}{w_0}=9$ (see Corollary \ref{corimmediat}).
%\end{ex}

\begin{ex} Let us consider the function $f:(\C^4,0)\to (\C,0)$ of \cite[Example 4.12]{BiviaEncinas2011}, that is
$f(x,y,z,t)=z^9-y^{11}t+yt^5+x^{27}$. This function is weighted homogeneous of degree $27$ with respect to $w$, where $w=(1,2,3,5)$.
We observe that $A_w=\{1\}$ and $\supp_1(f)=\{(27,0,0,0)\}$. Then $\LL_0(\nabla f)=26$, by Corollary \ref{sufficient}.
In \cite{BiviaEncinas2011} we arrived to the same conclusion via the notion of $w$-matching.
\end{ex}

If $h\in\O_n$, then we denote by $J(h)$ the ideal of $\O_n$ generated by $\frac{\partial h}{\partial x_1},\dots,\frac{\partial h}{\partial x_n}$.

\begin{ex} Let $f:(\C^3,0)\to (\C,0)$ be the function defined by $$f(x,y,z)=x^{12}+y^4+z^3+x^6yz.$$
We have that $f$ is semi-weighted homogeneous of degree $d=12$ with respect to $w=(1,3,4)$. Moreover
$f$ satisfies the hypothesis of Corollary \ref{sufficient}. Then
$$
\LL_0(p_w(f))=\LL_0(\nabla f)=\frac{d-w_0}{w_0}=11.
$$
We remark that $f=p_w(f)+x^6yz$ and $d_w(x^6yz)=13>d$. That is, $f$ is not weighted homogeneous with respect to $w$, hence
it is not possible to apply the results of \cite{KOP2009} in order to obtain the value of $\LL_0(\nabla f)$.
Moreover, since $\Gamma_+(x^6z) \nsubseteq\Gamma_+(m\cdot J(p_w(f)))$ we have $J(x^6yz)\nsubseteq \overline{m\cdot J(p_w(f))}$ and therefore
the equality $\LL_0(\nabla f)=\LL_0(\nabla p_w(f))$ is not a consequence of Proposition \ref{Ig}.
\end{ex}

\begin{ex} Let us fix an integer $a\geq 3$. Let us consider the function $h:(\C^4,0)\to (\C,0)$ given by
$$
h(x,y,z,t)=xy^{a+2}+z^a+tx^{a-2}+t^{a-1}.
$$
A straightforward computation shows that $h$ has an isolated singularity at the origin and that $h$
is weighted homogeneous of degree $d=a^3+a^2-2a$ with respect to the vector of weights $w=(a^2+2a, a^2-2a, a^2+a-2, a^2+2a)$.
Hence we have $\mu(h)=a^4-3a^3+2a^2$, as a consequence of the Milnor-Orlik formula \cite{MO} (or more generally, as a consequence
of Proposition \ref{BezoutlikeNewton} and Lemma \ref{oneface}).
Since $a\geq 3$, the minimum of the components of $w$ is $a^2-2a$ and $A_w=\{2\}$.
Hence $h$ satisfies the hypothesis of Corollary \ref{sufficient} and therefore
$$
\LL_0(\nabla h)=\frac{a^3+a^2-2a-(a^2-2a)}{a^2-2a}=\frac{a^2}{a-2}.
$$
Moreover, also from Corollary \ref{sufficient}, we have $\LL_0(\nabla f)=\LL_0(\nabla h)$, for all $f\in\O_4$
such that $p_w(f)=h$.
\end{ex}

%\begin{ex} Let us consider the function $f:(\C^4,0)\to (\C,0)$ given by $$f(x,y,z,t)=xy^7+z^5+tx^3+t^4.$$
%A straightforward computations shows that $f$ has an isolated singularity at the origin and that $f$
% is weighted homogeneous of degree $d=140$ with respect to the vector of weights $w=(35, 15, 28, 35)$.
%We observe that $f$ satisfies the hypothesis of Corollary \ref{sufficient}, then $\LL_0(\nabla f)=\frac{140-15}{15}=\frac{25}{3}$.
%\end{ex}

%%%%%%%%%%%%%%%%%%%%%%%%%%%%%%%%%%%%%%%%%%%%%%%%%%%%%%%%%%%%%%%%%%%%%%%%%%%%%%%%%%%%%%%%%%%%%%%%%%%%%%%%%%%%%%%%%%%%%%%%%%%%%%%

\end{document}